\documentclass[11pt]{article}
\RequirePackage[OT1]{fontenc}
\RequirePackage{amsthm,amsmath,natbib}
\RequirePackage[colorlinks]{hyperref}
\usepackage{amsfonts}
\usepackage{amssymb}
\usepackage{amstext}
\usepackage{parskip}
\usepackage{fullpage}
\bibpunct{(}{)}{;}{a}{,}{,}

\usepackage{times}
\usepackage{graphicx}
\usepackage{epsfig}

\newcommand{\E}{\ensuremath{\mathbb E}}
\newcommand{\R}{\ensuremath{\mathbb R}}

\newcommand{\Z}{\ensuremath{\mathbb Z}}
\newcommand{\N}{\ensuremath{\mathbb N}}

\newcommand{\prob}[1]{\ensuremath{{\mathbb P}\left(#1\right)}}

\newcommand{\expct}[1]{\ensuremath{{\mathbb E}#1}}

\newcommand{\size}[1]{\ensuremath{\left|#1\right|}}

\newcommand{\argmin}{\operatorname{argmin}}

\newcommand{\RE}{\textnormal{\textsf{RE}}}

\def \a {\alpha}

\def \e {\varepsilon}
\def \d {\delta}
\def \D {\Delta}

\def \r {\rho}

\def \S {\Sigma}
\def \t {\tau}

\newcommand{\ve}{\varepsilon}

\newcommand{\silent}[1]{}

\newcommand{\ip}[1]{\langle{\,#1\,}\rangle}

\newcommand{\Ball}{{B}}

\newcommand{\M}{{\mathcal M}}

\newcommand{\Net}{{\mathcal N}}

\newcommand{\sign}{\text{\rm sgn}}

\newcommand{\inv}[1]{\frac{1}{#1}}
\newcommand{\abs}[1]{\left\lvert#1\right\rvert}
\newcommand{\twonorm}[1]{\left\lVert#1\right\rVert_2}

\newcommand{\shtwonorm}[1]{\lVert#1\rVert_2}
\newcommand{\norm}[1]{\left\lVert#1\right\rVert}

\def \etc {,\ldots,}

\def \absconv {{\rm abs.conv}}

\def \id {{\it id}}

\newcommand{\pr}[2]{\langle {#1} , {#2} \rangle}

\newcommand{\bens}{\begin{eqnarray*}}
\newcommand{\eens}{\end{eqnarray*}}
\newcommand{\ben}{\begin{eqnarray}}
\newcommand{\een}{\end{eqnarray}}
\newcommand{\bnum}{\begin{enumerate}}
\newcommand{\enum}{\end{enumerate}}
\newcommand{\bit}{\begin{itemize}}
\newcommand{\eit}{\end{itemize}}

\newcommand{\beq}{\begin{equation}}
\newcommand{\eeq}{\end{equation}}

\def\P{{\mathbb P}}

\newcommand{\Cone}{{\rm Cone}}

\def\supp{\mathop{\text{\rm supp}\kern.2ex}}
\def\conv{\mathop{\text{\rm conv}\kern.2ex}}
\def\span{\mathop{\text{\rm span}\kern.2ex}}
\def\absconv{\mathop{\text{\rm absconv}\kern.2ex}}
\def\argmin{\mathop{\text{arg\,min}\kern.2ex}}

\let\hat\widehat
\def\W{\Cone}

\numberwithin{equation}{section}
\theoremstyle{plain}
\newtheorem{theorem}{Theorem}[section]

\newtheorem{lemma}[theorem]{Lemma}

\newtheorem{definition}[theorem]{Definition}

\newtheorem{remark}[theorem]{Remark}
\newtheorem{corollary}[theorem]{Corollary}

{\end{eqnarray}\end{subequations}\hskip-4.0pt}%
\def\math{${}^\diamond$}
\def\stat{${}^\star$}
\newenvironment{proofof}[1]{\hspace*{20pt}{\it Proof}{ of #1}.\hskip10pt}{\qed\vskip5pt}
\newenvironment{proofof2}{}{\qed\vskip5pt}
\begin{document}
\title{Reconstruction from anisotropic random measurements
}

\author{Mark Rudelson\math \footnote{
Research was supported in part by  NSF grant DMS-0907023.}
 \; and\; Shuheng Zhou\stat\;\footnote{emails: rudelson@umich.edu, shuhengz@umich.edu} \\
{\math}Department of Mathematics,\\
{\stat}Department of Statistics,\\
University of Michigan, Ann Arbor, MI 48109-1107}

\maketitle

\begin{abstract}
Random matrices are widely used in sparse recovery problems,
and the relevant properties of matrices with i.i.d. entries are
well understood. The current paper discusses the recently introduced
Restricted Eigenvalue (RE) condition, which is among the most general
assumptions on the matrix, guaranteeing recovery. We prove a reduction
principle showing that the RE condition can be guaranteed by
checking the restricted isometry on a certain family of low-dimensional
subspaces. This principle allows us to establish the RE condition
for several broad classes of random matrices with dependent entries,
including random matrices with subgaussian rows and non-trivial covariance
structure, as well as matrices with independent rows, and uniformly bounded entries.
\end{abstract}

\section{Introduction}
\label{sec:intro}
In a typical high dimensional setting, the number of variables $p$ is much
larger than the number of observations $n$. This challenging setting appears
in statistics and signal processing, for example, in regression,
covariance selection on Gaussian graphical models,
signal reconstruction, and sparse approximation.
Consider a simple setting, where we try to recover a vector $\beta \in \R^p$
in the following linear model:
\beq
\label{eq::linear-model}
Y = X \beta + \epsilon.
\eeq
Here $X$ is an $n \times p$ design matrix, $Y$ is a vector
of noisy observations, and $\epsilon$ is the noise term.
Even in the noiseless case, recovering $\beta$ (or its support)
from $(X, Y)$ seems impossible when $n \ll p$, given that we have more
variables than observations.

A line of recent research shows that when $\beta$ is sparse,
that is, when it has a relatively small number of
nonzero coefficients,
 it is possible to recover $\beta$ from an underdetermined system of equations.
In order to ensure reconstruction, the design matrix $X$ needs to
behave sufficiently nicely in a sense that it satisfies certain incoherence conditions.
One notion of the incoherence which has been formulated in
the sparse reconstruction literature~\citep{CT05,CT06,CT07} bears the name of
Uniform Uncertainty Principle (UUP).
It states that for all $s$-sparse sets $T$, the matrix $X$ restricted to the columns from $T$
acts as an almost isometry.
Let $X_T$, where $T \subset \{1, \ldots, p\}$ be the $n \times |T|$
submatrix obtained by extracting columns of $X$ indexed by $T$.
For each integer $s  =1, 2, \ldots$ such that $s < p$,
the $s$-restricted isometry constant $\theta_s$ of
$X$ is the smallest quantity such that
\ben
\label{eq::RIP}
(1 - \theta_s) \twonorm{c}^2 \leq \twonorm{X_T c}^2/n
\leq (1 + \theta_s) \twonorm{c}^2,
\een
for all $T \subset \{1,\ldots, p\}$ with $|T| \leq s$ and coefficients
sequences $(c_j)_{j \in T}$.
Throughout this paper, we refer to a vector $\beta \in \R^{p}$ with at most
$s$ non-zero entries, where $s \leq p$, as a {\bf $s$-sparse} vector.

To understand the formulation of the UUP, consider the simplest noiseless case as
mentioned earlier, where we assume $\epsilon = 0$ in~\eqref{eq::linear-model}.
Given a set of values $(\ip{X^i, \beta})_{i=1}^n$,
where $X^1, X^2, \ldots, X^n$ are independent random vectors in $\R^p$,
the basis pursuit program{~\citep{Chen:Dono:Saun:1998} finds $\hat\beta$
which minimizes the $\ell_1$-norm of $\beta'$ among all $\beta'$
satisfying $X \beta' = X \beta$, where $X$ is a $n \times p$ matrix
with rows $X^1, X^2, \ldots, X^n$.
This can be cast as a linear program and thus is computationally efficient.
Under variants of such conditions, the exact recovery or approximate
reconstruction of a
sparse $\beta$ using the basis pursuit program has been shown in a series of
powerful results~\citep{Donoho:cs,Donoho:04,CRT06,CT05,CT06,Donoho06,RV06,RV08,CT07}.
We refer to these papers for further
references on earlier results for sparse recovery.

In other words, under the UUP, the design matrix $X$ is taken as a $n \times p$
measurement ensemble through which one aims to recover both
the unknown non-zero positions and the strength of a $s$-sparse
signal $\beta$ in $\R^p$ efficiently (thus the name for compressed sensing).
Naturally, we wish $n$ to be as small as possible for given values of $p$ and $s$.
It is well known that for random matrices, UUP holds for $s = O(n/\log(p/n))$ with
i.i.d. Gaussian random entries,  Bernoulli, and in general  subgaussian
entries~\citep{CT05,RV05,CT06,Donoho06,BDDW08,MPT08}.
Recently, it has been shown~\citep{ALPT09} that UUP holds for
$s = O(n/\log^2(p/n))$ when $X$ is a random matrix composed of
columns that are independent isotropic vectors with log-concave densities.
For a random Fourier ensemble, or randomly sampled rows of orthonormal
matrices, it is shown that~\citep{RV06,RV08} the UUP holds for
$s = O(n/\log^c p)$ for $c=4$, which improves upon the earlier result of~\cite{CT06}
where $c =6$.
To be able to prove UUP for random measurements or design matrix,
the isotropicity condition (cf. Definition~\ref{def:psi2-vector}) has been assumed
in all literature cited above.
This assumption is not always reasonable in statistics and machine learning,
where we often come across high dimensional data with correlated entries.

The work of~\cite{BRT09} formulated the restricted eigenvalue (RE) condition and
showed that it is among the weakest and hence the most general conditions in
literature imposed on the Gram matrix in order to guarantee nice statistical properties
for the Lasso estimator~\citep{Tib96} as well as the Dantzig selector~\citep{CT07}.
In particular, it is shown to be a relaxation of the UUP under suitable choices
of parameters involved in each condition; see~\cite{BRT09}.
We now state one version of the
{\em Restricted Eigenvalue condition} as formulated in ~\citep{BRT09}.
For some integer $0< s_0 <p$ and a positive number $k_0$,
$\RE(s_0, k_0, X)$ for matrix $X$ requires that the following holds:
\beq
\label{eq::admissible}
\forall \upsilon \not = 0, \; \;
\min_{\stackrel{J \subset \{1, \ldots, p\},}{|J| \leq s_0}}
\min_{\norm{\upsilon_{J^c}}_1 \leq k_0 \norm{\upsilon_{J}}_1}
\; \;  \frac{\norm{X \upsilon}_2}{\norm{\upsilon_{J}}_2} > 0,
\eeq
where $\upsilon_{J}$ represents the subvector of $\upsilon \in \R^p$
confined to a subset $J$ of $\{1, \ldots, p\}$.
In the context of compressed sensing, RE condition can also be taken as a way to
guarantee recovery for anisotropic measurements.
We refer to~\cite{GB09} for other conditions which are closely related to the RE condition.

Consider now the linear regression model in~\eqref{eq::linear-model}.
For a chosen penalization parameter $\lambda_n \geq 0$, regularized
estimation with the $\ell_1$-norm penalty, also known as the
Lasso \citep{Tib96} refers to the following convex optimization problem
\begin{eqnarray}
\label{eq::origin} \; \;
\hat \beta = \arg\min_{\beta} \frac{1}{2n}\|Y-X\beta\|_2^2 +
\lambda_n \|\beta\|_1,
\end{eqnarray}
where the scaling factor $1/(2n)$ is chosen for convenience.
Under i.i.d Gaussian noise and the RE condition,
bounds on $\ell_2$ prediction loss and on $\ell_q$,
 $1 \leq q \leq 2$, loss for estimating the parameter $\beta$ in \eqref{eq::linear-model} for
both the Lasso and the Dantzig selector have all been derived in~\cite{BRT09}.
In particular, $\ell_2$ loss of
$\Theta(\lambda  \sigma \sqrt{s})$ were obtained for the Lasso
under $\RE(s, 3, X)$ and the Dantzig selector under $\RE(s, 1, X)$
respectively in~\cite{BRT09}, where it is shown that $\RE(s, 1, X)$
condition is weaker than the UUP used in~\cite{CT07}.

RE condition with parameters $s_0$ and $k_0$
for random measurements / design matrix has been proved for a
random Gaussian vector~\cite{RWY09,RWY10} with a sample bound of
order $n = O(s_0 \log p)$, when condition~\eqref{eq::admissible} holds for
the square root of the population covariance matrix $\Sigma$.
As we show below, the bound $n = O(s_0 \log p)$ can be improved to the optimal one
$n = O(s_0 \log (p/s_0))$ when $\RE(s_0, k_0, \Sigma^{1/2})$ is replaced
with $\RE(s_0, (1+\ve)k_0, \Sigma^{1/2})$ for any $\ve > 0$.
The papers ~\cite{RWY09,RWY10} have motivated the investigation for a non-iid
subgaussian random design by~\cite{Zhou09c}, as well as the present work.
The proof of~\cite{RWY10} relies on a deep result from the theory of Gaussian
random processes -- Gordon's Minimax Lemma~\cite{Gor85}.
However, this result relies on the properties of the normal random variables,
and is not available beyond the Gaussian setting. To 
establish the RE condition for more general classes of random matrices
we had to introduce a new approach based on geometric functional analysis.
We defer the comparison of the present paper with~\cite{Zhou09c} to
Section~\ref{sec:random-RE}.
Both~\cite{ZGB09} and~\cite{GB09} obtained weaker results which are
based on bounding the maximum entry-wise difference between sample
and the population covariance matrices.
We refer to~\cite{RWY10} for a more elaborate comparison.

\subsection{Notation and definitions}

Let $e_1, \ldots, e_p$ be the canonical basis of $\R^p$.
For a set $J \subset \{1, \ldots, p\}$, denote
$E_J = \span\{e_j: j \in J\}$.
For a matrix $A$, we use $\twonorm{A}$ to denote its operator norm.
For a set $V \subset \R^p$,
we let $\conv V$ denote the convex hull of $V$. For a finite set
$Y$, the cardinality is denoted by $|Y|$. Let $\Ball_2^p$ and
$S^{p-1}$ be the unit Euclidean ball and the unit sphere respectively.
For a vector $u \in \R^p$, let $u_{T_0}$ be the subvector of $u$
confined to the locations of its $s_0$ largest coefficients in absolute values.
In this paper, $C, c$, etc, denote various absolute constants which may
change line by line.
Occasionally, we use $u_T \in \R^{|T|}$, where
$T \subseteq \{1, \ldots, p\}$, to also represent its $0$-extended
version $u' \in \R^p$ such that $u'_{T^c} =0$ and
 $u'_{T} =u_T$.

We define $\W(s_0, k_0)$, where $0 < s_0 < p$ and $k_0$ is a positive number,
as the set of vectors in $\R^p$
which satisfy the following cone constraint:
\ben
\label{eq::cone-init}
\W(s_0, k_0) = \left\{x \in \R^p \;|\; \exists I \in \{1, \ldots, p\}, \size{I} = s_0
\; \mbox{ s.t. } \; \norm{x_{I^c}}_1 \leq k_0 \norm{x_{I}}_1 \right\}.
\een
Let $\beta$ be a $s$-sparse vector and $\hat{\beta}$ be the solution
from either the Lasso or the Dantzig selector.
One of the common properties of the Lasso and the Dantzig
selector is: for an appropriately chosen $\lambda_n$ and
under i.i.d. Gaussian noise, the condition
\ben
\label{eq::cone}
\upsilon := \hat{\beta} - \beta  \in \W(s, k_0)
\een
holds with high probability.
Here $k_0=1$ for the Dantzig selector, and $k_0 = 3$ for the Lasso;
see~\cite{BRT09} and~\cite{CT07} for example.
The combination of the cone property~\eqref{eq::cone} and the RE condition
leads to various nice convergence results as stated earlier.

We now define some parameters related to the RE and sparse eigenvalue conditions
that are relevant.
\begin{definition}
\label{def:memory}
Let $1 \leq s_0 \leq p$, and let $k_0$ be a positive number.
We say that a $q \times p$ matrix $A$ satisfies $\RE(s_0, k_0, A)$
 condition with parameter $K(s_0, k_0, A)$ if for any $\upsilon
 \not=0$,
\beq
\label{eq::admissible-random}
\inv{K(s_0, k_0, A)} := \min_{\stackrel{J \subseteq \{1, \ldots,
    p\},}{|J| \leq s_0}}
\min_{\norm{\upsilon_{J^c}}_1 \leq k_0 \norm{\upsilon_{J}}_1}
\; \;  \frac{\norm{A \upsilon}_2}{\norm{\upsilon_{J}}_2} > 0.
\eeq
\end{definition}
It is clear that when $s_0$ and $k_0$ become smaller,
this condition is easier to satisfy.
\begin{definition}
\label{def::sparse-eigen}
For $m \leq  p$, we define the largest and smallest
$m$-sparse eigenvalue of a $q \times p$ matrix $A$ to be
\ben
\label{eq::eigen-Sigma}
\rho_{\max}(m, A) & := &
\max_{t \not= 0; m-\text{sparse}} \; \;\shtwonorm{A t}^2/\twonorm{t}^2, \\
\label{eq::eigen-Sigma-min}
\rho_{\min}(m, A) & := &
\min_{t \not= 0; m-\text{sparse}} \; \;\shtwonorm{A t}^2/\twonorm{t}^2.
\een
\end{definition}
\subsection{Main results}
\label{sec:random-RE}
The main purpose of this paper is to show that the RE condition holds with high
probability for systems of random measurements/random design
matrices of a general nature. To establish such result with high
probability, one has to assume that it holds in average. So, our
problem boils down to showing that, under some assumptions on random variables,
the RE condition on the covariance matrix implies a similar condition on a
random design matrix with high probability when $n$ is sufficiently large
(cf. Theorems~\ref{thm:subgaussian-T-intro} and
Theorem~\ref{thm::RE-bounded-entries-intro}).
This generalizes the results on UUP mentioned above,
where the covariance matrix is assumed to be identity.

Denote by $A$ a fixed $q \times p$ matrix. We consider the
design matrix $X$ which can be represented as
\ben
\label{eq::rand-des}
X= \Psi A,
\een
where the rows of the matrix $\Psi$ are isotropic random vectors.
An example of such a random matrix $X$ consists of independent rows,
each being a random vector in $\R^p$
that follows a multivariate normal distribution $N(0, \Sigma)$,
when we take $A = \S^{1/2}$ in~\eqref{eq::rand-des}.
Our first main result is related to this setup. We consider a matrix
represented as $\tilde{X}=\tilde{\Psi} A$, where the matrix $A$ satisfies the RE
condition. The result is purely geometric, so we consider
a {\em deterministic} matrix $\tilde{\Psi}$.

We prove a general reduction principle showing that if
the matrix $\tilde{\Psi}$ acts as almost isometry on the images of
the sparse vectors under $A$, then the product $\tilde{\Psi} A$
satisfies the RE condition with a smaller parameter $k_0$. More
precisely, we prove Theorem~\ref{thm::isometry-intro}.
\silent{
Define for a matrix $A$ which satisfies the $\RE(s_0, k_0, A)$ condition, for
a given integer $s_0$, and a positive number $k_0$,
\ben
\label{eq::sparse-dim-A}
d(k_0, A) & = & s_0 + s_0 \max_j  \twonorm{A e_{j}}^2 \frac{16 K^2(s_0, k_0, A) (k_0)^2 (k_0 + 1)}{\delta^2}
\een
}
\begin{theorem}
\label{thm::isometry-intro}
Let $1/5 > \delta > 0$.
Let $0 < s_0 < p$ and $k_0>0$.
Let $A$ be a $q \times p$ matrix such that
$\RE(s_0, 3k_0, A)$ holds for $0< K(s_0, 3k_0, A) < \infty$.
Set 
\ben
\label{eq::sparse-dim-A}
d & = & s_0 + s_0 \max_j  \twonorm{A e_{j}}^2 \frac{16 K^2(s_0, 3k_0, A) (3k_0)^2 (3k_0 + 1)}{\delta^2},
\een
and let $E=\cup_{|J| = d} E_J$ for $d< p$ and $E$ denotes  $\R^p$ otherwise.
Let $\tilde\Psi$ be a matrix such that
 \ben
\label{eq::sparse-isometry-intro}
   \forall x \in A E  \quad (1-\delta) \norm{x}_2 \le \norm{\tilde\Psi x}_2 \le
   (1+\delta) \norm{x}_2.
 \een
Then $\RE(s_0, k_0, \tilde{\Psi} A)$ condition holds for
matrix $\tilde{\Psi} A$ with $0< K(s_0, k_0, \tilde{\Psi} A)
\leq K(s_0, k_0, A)/(1-5 \delta)$.
\end{theorem}
\begin{remark}
We note that this result does not involve $\rho_{\max}(s_0, A)$, nor
the global parameters of the matrices $A$ and $\tilde{\Psi}$,
such as the norm or the smallest singular value.
We refer to~\cite{RWY10} for examples of matrix $A$,
where $\rho_{\max}(s_0,A)$ grows with $s_0$ while the RE condition still holds for $A$.
\end{remark}

The assumption $\RE(s_0, 3k_0, A)$ can be replaced by  $\RE(s_0, (1+\ve) k_0, A)$
for any $\ve > 0$ by appropriately increasing $d$. See Remark~\ref{remark::RE}
for details.

We apply the reduction principle to analyze different
classes of random design matrices. This analysis is reduced to
checking that the almost isometry property holds for all vectors
from some low-dimensional subspaces, which is easier than
checking the RE property directly.

The first example is the matrix $\Psi$ whose rows are
independent isotropic vectors with {\em subgaussian} marginals as in
Definition~\ref{def:psi2-vector}.
This result extends a theorem of~\cite{RWY10} to a non-Gaussian
setting, in which the entries of the design matrix may even not
have a density.
\begin{definition}
\label{def:psi2-vector}
Let $Y$ be a random vector in $\R^p$
\bnum
\item
$Y$ is called isotropic
if for every $y \in \R^p$, $\expct{\abs{\ip{Y, y}}^2} = \twonorm{y}^2$.
\item
$Y$ is $\psi_2$ with a constant $\alpha$ if for every $y \in \R^p$,
\beq
\norm{\ip{Y, y}}_{\psi_2} := \;
\inf \{t: \expct{\exp(\ip{Y,y}^2/t^2)} \leq 2 \}
\; \leq \; \alpha \twonorm{y}.
\eeq
\enum
\end{definition}
The  $\psi_2$ condition on a scalar random variable $V$ is equivalent to
the subgaussian tail decay of $V$, which means
\bens
\prob{|V| >t} \leq 2 \exp(-t^2/c^2), \; \; \text{for all} \; \; t>0.
\eens
Throughout this paper, we use $\psi_2$, vector with subgaussian marginals
and subgaussian vector interchangeably.
Examples of isotropic random vectors with subgaussian marginals are:
\begin{itemize}
\item  The random vector $Y$ with i.i.d $N(0,1)$ random coordinates.
\item Discrete Gaussian vector, which is a random vector taking
values on the integer lattice $\Z^p$ with distribution $\P(X = m) = C \exp(-\twonorm{m}^2/2)$
for $m \in \Z^p$.
\item A vector with independent centered bounded random coordinates.
The subgaussian property here follows from the Hoeffding inequality for
sums of independent random variables. This example includes,
in particular, vectors with random Bernoulli coordinates,
in other words, random vertices of the discrete cube.
\end{itemize}
It is hard to argue that such multivariate Gaussian or Bernoulli
random designs are not relevant for statistical applications.
\begin{theorem}
\label{thm:subgaussian-T-intro}
Set $0< \d < 1$,  $k_0 > 0$, and $0< s_0 < p$.
Let $A$ be a $q \times p$ matrix satisfying $\RE(s_0, 3k_0, A)$ condition
as in Definition~\ref{def:memory}.
Let $d$ be as defined in~\eqref{eq::sparse-dim-A},
and let $m = \min(d, p)$.
Let $\Psi$ be an $n \times q$ matrix whose rows are
independent isotropic $\psi_2$ random vectors in $\R^q$ with constant $\alpha$.
Suppose the sample size satisfies
\ben
\label{eq::UpsilonSampleBound-intro}
n \geq \frac{2000 m \alpha^4}{\d^2} \log \left(\frac{60 e p}{m \d}\right).
\een
Then with probability at least $1- 2 \exp(\d^2 n/2000 \alpha^4)$,
$\RE(s_0, k_0, (1/\sqrt{n})\Psi A)$ condition holds for matrix $(1/\sqrt{n}) \Psi A$
with
\ben
\label{eq::RE-subg}
0< K(s_0, k_0,  (1/\sqrt{n}) \Psi A) \leq \frac{K(s_0, k_0, A)}{1-\d}.
\een
\end{theorem}

\begin{remark}
We note that all constants in Theorem~\ref{thm:subgaussian-T-intro}
are explicit, although they are not  optimized.
\end{remark}
Theorem~\ref{thm:subgaussian-T-intro} is applicable in various contexts.
We describe two examples.
The first example concerns cases which have been considered
in~\cite{RWY10,Zhou09c}. They show that the RE condition on the covariance matrix
$\Sigma$ implies a similar condition on a
random design matrix $X = \Psi \Sigma^{1/2}$ with high probability when
$n$ is sufficiently large.
In particular, in~\cite{Zhou09c},
the author considered subgaussian random matrices of the form
$X = \Psi \Sigma^{1/2}$
where $\Sigma$ is a $p \times p$ positive semidefinite matrix
satisfying $\RE(s_0, k_0, \Sigma^{1/2})$ condition,
and $\Psi$ is as in Theorem \ref{thm:subgaussian-T-intro}.
Unlike the current paper, the author allowed $\rho_{\max}(s_0, \Sigma^{1/2})$
as well as $K^2(s_0, k_0, \Sigma^{1/2})$ to appear in the lower bound on $n$, and
showed that $X/\sqrt{n}$ satisfies the RE condition as in~\eqref{eq::RE-subg}
with overwhelming probability
whenever
\ben
\label{eq::sample-size-gen}
n >  \frac{9c' \alpha^4}{\delta^2}
(2 + k_0)^2 K^2(s_0, k_0, \Sigma^{1/2}) \min(4\rho_{\max}(s_0,\S^{1/2}) s_0 \log (5e p/s_0), s_0 \log p)
\een
where the first term was given in~\citet[Theorem 1.6]{Zhou09a} explicitly,
and the second term is an easy consequence by combining arguments
in~\cite{Zhou09a} and~\cite{RWY10}.
Analysis there used Corollary 2.7 in~\cite{MPT07} crucially.

In the present work, we get rid of the dependency of the sample size
on $\rho_{\max}(s_0, \S^{1/2})$, although under a slightly stronger
$\RE(s_0, 3k_0, \Sigma^{1/2})$ (See also Remark~\ref{remark::RE}).
More precisely, let $\Sigma$ be a $p \times p$ covariance matrix satisfying
$\RE(s_0, 3k_0, \Sigma^{1/2})$ condition.
Then,~\eqref{eq::RE-subg} implies that  with probability at
least $1- 2 \exp(\d^2 n/2000 \alpha^4)$,
\ben
\label{eq::RE-subg-special}
  0< K(s_0, k_0,  (1/\sqrt{n}) \Psi  \Sigma^{1/2}) \leq
\frac{K(s_0, k_0, \Sigma^{1/2})}{1-\d}
\een
where $n$ satisfies \eqref{eq::UpsilonSampleBound-intro} for $d$  defined
in~\eqref{eq::sparse-dim-A}, with  $A$ replaced by $\Sigma^{1/2}$.

Another application of Theorem~\ref{thm:subgaussian-T-intro} is given
in~\cite{ZLW09}. The $q \times p$ matrix $A$ can be taken as a data matrix
with $p$ attributes (e.g., weight, height, age, etc), and
$q$ individual records. The data are compressed by a random
linear transformation $X = \Psi A$. Such transformations have
have been called ``matrix masking'' in the privacy literature
\citep{duncan:91}.
We think of $X$ as ``public,'' while $\Psi$, which is a
$n \times q$ random matrix, is private and only needed at
the time of compression.  However, even with $\Psi$ known,
recovering $A$ from $\Psi$ requires solving a highly
under-determined linear system and comes with information theoretic
privacy guarantees when $n \ll q$,
as demonstrated in~\cite{ZLW09}. On the other hand,
sparse recovery using $X$ is highly feasible given that the RE conditions
are guaranteed to hold by Theorem~\ref{thm:subgaussian-T-intro}
with a small $n$.
We refer to~\cite{ZLW09} for a detailed setup on regression using
compressed data as in~\eqref{eq::rand-des}.

The second application of the reduction principle is to the
design matrices with uniformly bounded entries. As we mentioned above, if the entries of such matrix are independent, then its rows are subgaussian. However, the independence of entries is not assumed, so
 the decay of the marginals can be arbitrary slow.
A natural example for compressed sensing would be measurements of random
Fourier coefficients, when some of the coefficients cannot be measured.
\begin{theorem}  
\label{thm::RE-bounded-entries-intro}
Let $0<\d<1$ and $0<s_0<p$. Let $Y \in \R^p$ be a random vector such that
$\norm{Y}_{\infty} \le M$ a.s and denote $\Sigma = \E Y Y^T$. 
Let $X$ be an $n \times p$ matrix,
whose rows $X_1 \etc X_n$ are independent copies of $Y$.
Let $\Sigma$ satisfy the $\RE(s_0, 3k_0, \Sigma^{1/2})$ condition
as in Definition~\ref{def:memory}.
Let $d$ be as defined in~\eqref{eq::sparse-dim-A}, where we replace $A$ with
$\Sigma^{1/2}$.
Assume that $d \leq p$ and $\r=\r_{\min}(d,\S^{1/2}) >0$.
Suppose the sample size satisfies for some absolute constant $C$
  \[
    n \ge \frac{C M^2 d  \cdot \log p}{\r \d^2} \cdot
      \log^3 \left( \frac{C M^2 d \cdot \log p}{\r \d^2 }\right).
  \]
Then with probability at least $1- \exp \left( -  \d \r n/(6 M^2 d) \right)$,
$\RE(s_0, k_0, X)$ condition holds for matrix $X/\sqrt{n}$ with
$0< K(s_0, k_0, X/\sqrt{n})) \leq K(s_0, k_0,\Sigma^{1/2})/(1-\d)$.
 \end{theorem}

\begin{remark}
Note that unlike the case of a random matrix with subgaussian marginals,
the estimate of Theorem \ref{thm::RE-bounded-entries-intro} contains the
minimal sparse singular value $\r$. We will provide an example illustrating
that this is necessary in Remark~\ref{remark::lower-bound}.
\end{remark}

We will prove Theorems~\ref{thm::isometry-intro},~\ref{thm:subgaussian-T-intro},
and~\ref{thm::RE-bounded-entries-intro} in Sections~\ref{sec:reduction-proof},
~\ref{sec:subgaussian-proof}, and~\ref{sec:bounded-proof} respectively.

We note that the reduction principle can be applied to other types of
random variables. One can consider the case of heavy-tailed
marginals. In this case the estimate for the images of sparse
vectors can be proved using the technique developed by~\cite{Ver11a,Ver11b}.
One can also consider random vectors with log-concave densities,
and obtain similar estimates following the methods of~\cite{ALPT09,ALLPT11}.
We leave the details for an interested reader.

To make our exposition complete, we will show
some immediate consequences in terms of statistical inference on high
dimensional data that satisfy such RE and sparse eigenvalue
conditions. We discuss in Section~\ref{sec:background}
some bounds for the Lasso estimator for such a subgaussian random ensemble.
In particular, bounds developed in the present paper can be applied
to obtain tight convergence results for covariance estimation for
a multivariate Gaussian model~\cite{ZRXB10}.

\subsection{Convergence rates in sparse recovery}
\label{sec:background}
Lasso and the Dantzig selector are both
well studied and shown to have provable nice statistical properties.
For results on variable selection, prediction error and $\ell_p$ loss,
where $1 \leq p \leq 2$ under various incoherence conditions,
see, for example~\cite{GR04,MB06,ZY06,BTW07c,CT07,Kol07,vandeG08,ZH08,Wai09,CP09,BRT09,CWX09,Kol09,MY09}.
As mentioned, the restricted eigenvalue (RE) condition as formulated by~\cite{BRT09}
are among the weakest and hence the most general conditions in
literature imposed on the Gram matrix in order to
guarantee nice statistical properties for the Lasso and the Dantzig
selector. For a comprehensive comparison between some of these conditions,
we refer to~\cite{GB09}.

For random design as considered in the present paper,
one can show that various oracle inequalities in terms of $\ell_2$ convergence
hold for the Lasso and the Dantzig selector as long as $n$ satisfies the lower
bounds above. Let $s = \size{\supp{\beta}}$ for
$\beta$ in~\eqref{eq::linear-model}.
Under $\RE(s, 9, \Sigma^{1/2})$, a sample size of $n = O(s \log (p/s))$
is sufficient for us to derive bounds corresponding to those
in~\citet[Theorem 7.2]{BRT09}.
As a consequence, we see that this setup requires
$\Theta(\log(p/s))$ observations per nonzero value in $\beta$
where $\Theta$ hides a constant depending on $K^2(s, 9, \Sigma^{1/2})$
for the family of random matrices with subgaussian marginals which satisfies
$\RE(s, 9, \Sigma^{1/2})$ condition.
Similarly, we note that for random matrix $X$ with a.s. bounded entries
of size $M$, $n = O(s M^2 \log p \log^3 (s \log p))$ samples are sufficient
in order to achieve accurate statistical estimation.
We say this is a  {\em linear or sublinear sparsity}.
For $p \gg n$,  this is a desirable property as it implies that accurate
statistical estimation is feasible given a very limited amount of data.

As another example, assume that $\rho_{\max}(s, \Sigma^{1/2})$ is a bounded constant
and $\rho_{\min}(s, \Sigma^{1/2}) > 0$.
We note that this slight restriction on $\rho_{\max}(s, \S^{1/2})$
allows one to derive
an oracle result on the $\ell_2$ loss as studied by~\cite{Donoho:94,CT07,Zhou09a,Zhou10}),
which we now elaborate.
Let $\epsilon \sim N(0, \sigma^2 I)$ in \eqref{eq::linear-model}.
Assume that $\RE(s_0, 12, \Sigma^{1/2})$ holds, where
$\Sigma_{ii} =1, \forall i$ and
$s_0$ is defined as the smallest integer such that
\ben
\label{eq::define-s0}
\sum_{i=1}^p \min(\beta_i^2, \lambda^2 \sigma^2) \leq
s_0 \lambda^2 \sigma^2, \text{ where } \; \lambda = \sqrt{ 2 \log p/n}.
\een
We note that as a consequence of this definition is: $|\beta_j| < \lambda \sigma$ for all
$j > s_0$, if we order $|\beta_1| \geq |\beta_2| ... \geq |\beta_p|$; see
~\cite{CT07}.
Hence $s_0$ essentially characterizes the number of significant
coefficients of $\beta$ with respect to the noise level $\sigma$.
Following analysis in~\cite{Zhou10}, one can show that
the Lasso solution satisfies
\ben
\label{eq::oracle-ell2-bound}
\shtwonorm{\hat{\beta} - \beta}^2 \asymp s_0 \lambda^2 \sigma^2,
\een
with overwhelming probability, as long as
\ben
\label{eq::sample-size}
n \geq C m \log(c p/m)
\een
where $m = \max(s, d)$ for $d$
as defined in~\eqref{eq::sparse-dim-A} with  $\Sigma^{1/2}$ replacing $A$.

One can also show the same bounds on $\ell_1$ loss and prediction error
as in~\cite{Zhou10} under this setting.
The rate of~\eqref{eq::oracle-ell2-bound} is an obvious
improvement upon the rate of $\Theta(\lambda  \sigma \sqrt{s})$
when $s_0$ is much smaller than $s$, that is, when
there are many non-zero but small entries in $\beta$.
Moreover, given such ideal rate on the $\ell_2$-loss, it is shown
in~\cite{Zhou09a,Zhou10} that one can then recover a sparse model
of size $\asymp 2s_0$ such that the model contains most of the important variables
while achieving such oracle inequalities as in \eqref{eq::oracle-ell2-bound},
where thresholding of the Lasso estimator followed by refitting has been applied.
Such results have also been used in Gaussian Graphical model selection to
show fast convergence rates in estimating the covariance matrix
and its inverse~\cite{ZRXB10}.

Conceptually, results in the current paper allow one to extend such oracle results
in terms of $\ell_2$ loss from the family of random matrices obeying the UUP
to a broader class of random matrices that satisfy the RE condition
with sample size at essentially the same order.
When $\Sigma$ is {\it ill-behaving} in the sense that $\rho_{\max}(m, \S^{1/2})$
grows too
rapidly as a function of $m$, we resort to the bound of $O(\lambda \sigma \sqrt{s})$
which corresponds to those derived in~\cite{BRT09}, under $\RE(s, 9, \Sigma)$.

Finally, the incoherence properties for a random design matrix
that is the composition of a random matrix with a deterministic matrix have
been studied even earlier, see for example~\cite{RSV08,ZLW09},
in the context of signal reconstruction and high dimensional sparse
regressions.

\section{Reduction principle}
\label{sec:reduction-proof}
We first reformulate the reduction principle in the form of
restrictive isometry:
we show that if the matrix $\tilde{\Psi}$ acts as almost isometry on
the images of the sparse vectors under $A$, then it acts the same way on
the images of a set of vectors which satisfy the cone
constraint~\eqref{eq::cone-init}.
We then prove Theorem~\ref{thm::isometry-intro} as a corollary of
Theorem~\ref{thm::isometry}.
\begin{theorem}
\label{thm::isometry}
Let $1/5 > \delta > 0$.
 Let $0 < s_0 < p$ and $k_0>0$.
Let $A$ be a $q \times p$ matrix such that
$\RE(s_0, 3k_0, A)$ condition holds for $0< K(s_0, 3k_0, A) < \infty$.
Set
\[
  d  = s_0 +  s_0 \max_j \twonorm{A e_j}^2 \left(\frac{16 K^2(s_0, 3k_0, A) (3k_0)^2 (3k_0 + 1)}{\delta^2}\right),
\]
and let $E=\cup_{|J| = d} E_J$ for $d<p$ and $E = \R^p$ otherwise.
Let $\tilde\Psi$ be a matrix such that
 \ben
\label{eq::sparse-isometry}
   \forall x \in A E  \quad (1-\delta) \norm{x}_2 \le \norm{\tilde\Psi x}_2 \le
   (1+\delta) \norm{x}_2.
 \een
Then for any $x \in A \Big(\W(s_0,k_0)\Big) \cap S^{q-1}$,
\ben
\label{eq::sparse-isometry-II}
   (1-5\delta) \le \norm{\tilde\Psi x}_2 \le (1+ 3 \delta)
 \een
\end{theorem}

\begin{proofof}{Theorem~\ref{thm::isometry-intro}}
By the $\RE(s_0, 3k_0, A)$ condition, $\RE(s_0, k_0, A)$
condition holds as well. 
Hence for $u \in \W(s_0,k_0)$ such that $u \not=0$,
\bens
\norm{ A u}_2 \ge \frac{\norm{u_{T_0}}_2}{K(s_0, k_0, A)} > 0,
 \eens
and by~\eqref{eq::sparse-isometry-II}
\bens
\norm{\tilde\Psi A u}_2 \ge (1-5\delta) \norm{ A u}_2 \ge
  (1-5\delta) \frac{\norm{u_{T_0}}_2}{K(s_0, k_0, A)} > 0.
\eens
\end{proofof}

The proof of Theorem~\ref{thm::isometry} uses several auxiliary results,
which will be established in the next two subsections.

\subsection{Preliminary results}

Our first lemma is based on Maurey's empirical  approximation argument~\cite{Pisi81}.
We show that any vector belonging to the convex hull of many vectors
can be approximated by a convex combination of a few of them.
\begin{lemma}
\label{lemma::maurey}
Let $u_1, \ldots, u_M \in \R^q$.
Let $y \in \conv(u_1, \ldots, u_M)$.
There exists a set $L \subset \{1, 2, \ldots, M\}$
such that
$$|L| \leq m
= \frac{4 \max_{j \in \{1, \ldots, M\}} \twonorm{u_j}^2}{\ve^2}$$
and a vector $y' \in \conv(u_j, j \in L)$
such that
$$\twonorm{y' - y} \leq \ve.$$
\end{lemma}

\begin{proof}
Assume that
$$y = \sum_{j \in \{1, \ldots, M\}} \alpha_j u_j\; \; \text{ where } \; \;
 \alpha_j \geq 0,  \text{ and } \;
\sum_{j} \alpha_j = 1.$$
Let $Y$ be a random vector in $\R^q$ such that
$$\prob{Y = u_{\ell}}  = \alpha_{\ell}, \; \ell \in \{1, \ldots, M\}$$
Then
$$\expct{Y} = \sum_{\ell \in \{1, \ldots, M\}} \alpha_{\ell} u_{\ell} = y.$$
Let $Y_1, \ldots, Y_m$ be independent copies of $Y$ and
let $\ve_1, \ldots, \ve_m$ be $\pm 1$ i.i.d. mean zero Bernoulli random
variables, chosen independently of $Y_1, \ldots, Y_m$.
By the standard symmetrization argument, we have
\ben
\label{eq::mean-bound}
\expct{\twonorm{y - \inv{m}\sum_{j=1}^m Y_j}^2} \leq
4 \expct{\twonorm{\inv{m}\sum_{j=1}^m \ve_j Y_j} }^2 =
\frac{4}{m^2} \sum_{j=1}^m \expct{\twonorm{Y_j}^2} \leq
  \frac{4 \max_{\ell \in \{1, \ldots, M\}} \twonorm{u_{\ell}}^2}{m}
\leq \ve^2
\een
where
$$\expct{\twonorm{Y_j}^2}
\leq \sup \twonorm{Y_j}^2 \leq \max_{\ell \in \{1, \ldots, M\}} \twonorm{u_{\ell}}^2
$$
and  the last inequality in~\eqref{eq::mean-bound} follows from the definition of $m$.

Fix a realization $Y_j = u_{k_j}$, $j=1, \ldots, m$ for which
\bens
\twonorm{y - \inv{m}\sum_{j=1}^m Y_j} \leq \ve.
\eens

The vector $\inv{m}\sum_{j=1}^m Y_j$ belongs to the convex hull of
$\{u_{\ell} \; : \; \ell \in L\}$, where
$L$ is the set of different elements from the sequence $k_1, \ldots, k_m$.
Obviously $|L| \leq m$ and the lemma is proved.
\silent{
One of these realizations must satisfy the property that
for some sequence $L'$, where repetition of elements are allowed, such that $|L'| = m$ and
\ben
\label{eq::maurey}
\inv{m}\sum_{i=1}^m Y_i = \inv{m}\sum_{\ell \in L'} \tilde{Z_{\ell}} & := & A u
\text{ such that }\; \;
\twonorm{y - A u} \leq \ve,
\een
In particular, let $J'$ be the set of unique elements of $L'$ with $\size{J'} \leq m$,
then $ u \in \span(e_j | j \in J')$ and
\ben
u = \inv{m}\sum_{\ell \in L'} \tilde{e}_{\ell} \; \; \text{ for } \;
\tilde{e}_{\ell} \in \{e_{\ell}, -e_{\ell},\}
\een
given that $\rho_{\min}(m, A) > 0$
and $A u = \inv{m}\sum_{\ell \in L'} \tilde{Z}_{\ell} = A \left(\inv{m}\sum_{\ell \in L'} \tilde{e}_{\ell}\right)$
where $A \tilde{e}_{\ell} : = \tilde{Z}_{\ell}, \forall \ell$.
(Suppose otherwise, it is impossible for~\eqref{eq::mean-bound} to hold.)
It is also clear that
\ben
\norm{u}_1 = \norm{\inv{m}\sum_{\ell \in L'} \tilde{e_{\ell}}}_1 = 1
\een}
\end{proof}

For each vector $x \in \R^p$, let ${T_0}$ denote the locations of the $s_0$
largest coefficients of $x$ in absolute values.
Any vector $x \in \W(s_0, k_0) \cap S^{p-1}$ satisfies:
\ben
\label{eq::init-norm-inf}
\norm{x_{T_0^c}}_{\infty}  \leq \norm{x_{T_0}}_{1}/s_0
& \leq & \frac{ \twonorm{x_{T_0}}}{\sqrt{s_0}}  \\
\label{eq::init-norm-1}
\norm{x_{T_0^c}}_{1}
 \leq   k_0 \sqrt{s_0} \twonorm{x_{T_0}}
& \leq & k_0 \sqrt{s_0} ; \text{ and } \; \twonorm{x_{T_0^c}} \leq 1.
\een
The next elementary estimate will be used in conjunction with the RE condition.
\begin{lemma}
\label{lemma::lower-bound-Az}
For each vector $\upsilon \in \W(s_0, k_0)$, let
${T_0}$ denotes the locations of the $s_0$
largest coefficients of $\upsilon$ in absolute values.  Then
\ben
\label{eq::cone-top-norm}
\twonorm{\upsilon_{T_{0}}} \geq \frac{\twonorm{v}}{\sqrt{1 + k_0}}.
\een
\end{lemma}

\begin{proof}
By definition of $\W(s_0, k_0)$, by~\eqref{eq::init-norm-inf}
\bens
\twonorm{\upsilon_{T_0^c}}^2
\leq \norm{\upsilon_{T_0^c}}_1 \norm{\upsilon_{T_0^c}}_{\infty}
 \leq
k_0 \norm{\upsilon_{T_0}}_1 \cdot \norm{\upsilon_{T_0}}_{1}/s_0 \leq
k_0 \twonorm{\upsilon_{T_0}}^2.
\eens
Therefore $\twonorm{\upsilon}^2  = \twonorm{\upsilon_{T_0^c}}^2  + \twonorm{\upsilon_{T_0}}^2
\leq ( k_0 + 1) \twonorm{\upsilon_{T_0}}^2.$
\end{proof}

The next lemma concerns the extremum of a linear functional on a
big circle of a $q$-dimensional sphere. We consider a line passing
through the extreme point, and show that the value of the functional
on a point of the line, which is relatively close to the extreme point,
provides a good bound for the extremum.
\begin{lemma}
\label{lemma::local-maxima}
let $u, \theta, x \in \R^q$ be vectors such that
\bnum
\item
$\twonorm{\theta} = 1$.
\item
$\ip{x, \theta} \not=0.$
\item
Vector $u$ is not parallel to $x$.
\enum
Define $\phi: \R \to \R$ by:
\ben
\phi(\lambda) = \frac{\ip{x + \lambda u, \theta} }{\twonorm{x + \lambda u}}.
\een
Assume $\phi(\lambda)$ has a local maximum at $0$, then
$$\frac{\ip{x + u, \theta}}{\ip{x, \theta}} \geq 1 - \frac{\twonorm{u}}{\twonorm{x}}.$$
\end{lemma}

\begin{proof}
Let $v = \frac{x}{\twonorm{x}}.$
Also let
\bens
\theta & = & \beta v + \gamma t, \text{ where } t \perp v, \twonorm{t} = 1 \; \text{ and }
\beta^2 + \gamma^2 = 1, \beta \not= 0 \\
\text{ and } \;
u & = & \eta v + \mu t + s \text{ where } s \perp v \; \text{ and }  s \perp t
\eens
Define $f: \R \to \R$ by:
\ben
f(\lambda) = \frac{\lambda}{\twonorm{x} + \lambda \eta}, \;
\lambda \not= -\frac{\eta}{\twonorm{x}}.
\een
Then
\bens
\phi(\lambda) & = & \frac{\ip{x + \lambda u, \theta} }{\twonorm{x + \lambda u}}
 =  \frac{\ip{( \twonorm{x} + \lambda \eta) v + \lambda \mu t + \lambda s,  \beta v + \gamma t}}
{\twonorm{(\twonorm{x} + \lambda \eta) v + \lambda \mu t + \lambda s}} \\
 & = & \frac{\beta(\twonorm{x} + \lambda \eta) + \lambda \mu \gamma}
{\sqrt{(\twonorm{x} + \lambda \eta)^2 + (\lambda \mu)^2 + \lambda^2 \twonorm{s}^2}} \\
& = & \frac{\beta + \mu \gamma f(\lambda)}{\sqrt{1 + (\mu^2 + \twonorm{s}^2) f^2(\lambda)}}
\eens
 Since
$f(\lambda) = \frac{\lambda}{\twonorm{x}} + O(\lambda^2)$
we have
$\phi(\lambda) = \beta + \mu\gamma \frac{\lambda}{\twonorm{x}} + O(\lambda^2)$  in the neighborhood of $0$,
Hence, in order to for $\phi(\lambda)$ to have a local maximum at $0$, $\mu$ or $\gamma$ must be 0. Consider these cases separately.
\bit
\item
First suppose $\gamma = 0$, then $\beta^2  =1$ and $\abs{\ip{x, \theta}} = \twonorm{x}$. Hence,
\bens
\frac{\ip{x + u, \theta}}{\ip{x, \theta}} = 1 + \frac{\ip{u, \theta}}{\ip{x, \theta}}
\geq 1 - \frac{|\ip{u, \theta}|}{\abs{\ip{x, \theta}}} \geq
1 - \frac{\twonorm{u}}{\twonorm{x}}
\eens
where $|\ip{u, \theta}| \leq \twonorm{u}$.
\item
Otherwise, suppose that $\mu = 0$. Then we have  $|\eta| = |\ip{u,v}| \leq \twonorm{u}$
and
\bens
\frac{\ip{x + u, \theta}}{\ip{x, \theta}}
= 1 + \frac{\ip{ \eta v + s, \beta v + \gamma t}}{\ip{v \twonorm{x}, \beta v + \gamma t}}
= 1 + \frac{\eta \beta}{\twonorm{x} \beta} = 1 + \frac{\eta}{\twonorm{x}}
\geq 1 - \frac{\twonorm{u}}{\twonorm{x}}
\eens
where we used the fact that $\beta \not= 0$ given $\ip{x, \theta} \not=0$.
\eit
\end{proof}

\subsection{Convex hull of sparse vectors}
For a set $J \subset \{1, \ldots, p\}$, denote
$E_J = \span\{e_j: j \in J\}$.
In order to prove the restricted isometry property of $\Psi$ over
the set of vectors in
$A \Big(\W(s_0,k_0)\Big) \cap S^{q-1}$, we first show that this set is contained in
the convex hull of the images of the sparse vectors with norms not
exceeding $(1-\d)^{-1}$. More precisely, we prove the following lemma.
\begin{lemma}
\label{lemma::sparse-approx}
Let $1 > \delta > 0$.
Let $0 < s_0 < p$ and $k_0>0$.
Let $A$ be a $q \times p$ matrix such that $\RE(s_0, k_0, A)$ condition
holds for $0< K(s_0, k_0, A) < \infty$.
Define
\ben \label{eq::definition-d(k_0,A)}
d = d(k_0,A) = s_0 +  s_0 \max_j \twonorm{A e_j}^2 \left(\frac{16 K^2(s_0, k_0, A) k^2_0 (k_0 + 1)}{\delta^2}\right).
\een
Then
\ben
\label{eq::convexity}
A \Big(\W(s_0,k_0)\Big) \cap S^{q-1} \subset
(1 -  \delta)^{-1} \conv\left(\bigcup_{\size{J} \leq d} A E_J \cap S^{q-1}\right)
\een
where for $d \geq p$, $E_J$ is understood to be $\R^p$.
\end{lemma}

\begin{proof}
Without loss of generality, assume that $d(k_0,A)<p$, otherwise the lemma is vacuously true.
For each vector $x \in \R^p$, let ${T_0}$ denote the locations of the $s_0$
largest coefficients of $x$ in absolute values.
Decompose a vector $x \in  \W(s_0, k_0) \cap S^{p-1}$ as
$$x = x_{T_0} + x_{T_0^c} \in x_{T_0} + k_0 \norm{x_{T_0}}_1 \absconv(e_j \; | \; j\in T_0^c), \;
\text{where } \twonorm{x_{T_0}} \geq \inv{\sqrt{k_0+1}} \text{ by~\eqref{eq::cone-top-norm}}$$
and hence
$$A x \in A x_{T_0} + k_0 \norm{x_{T_0}}_1 \absconv(A e_j \; | \; j\in T_0^c).$$
 Since the set $ A \Cone(s_0,k_0) \cap S^{q-1}$ is not easy to analyze, we introduce set of a simpler structure instead.
Define
\[
V = \left\{x_{T_0} + k_0 \norm{x_{T_0}}_1 \absconv(e_j \; | \; j\in T_0^c) | x \in \W(s_0, k_0) \cap S^{p-1} \;
\right\}.
\]
For a given $x  \in \W(s_0, k_0) \cap S^{p-1}$, if $T_0$ is not uniquely defined, we
include all possible sets of $T_0$ in the definition of $V$.
Clearly $V \subset \W(s_0,k_0)$ is a compact set. Moreover, $V$ contains a base of  $\W(s_0, k_0)$, that is, for any
$y \in \W(s_0, k_0) \setminus \{0\}$ there exists $\lambda > 0$ such that $\lambda y \in V$.

For any $v \in \R^p$ such that $\twonorm{Av} \not= 0$, define
$$F(v) = \frac{A v}{\twonorm{A v}}.$$
By condition $\RE(s_0,k_0,A)$, the function $F$ is well-defined and continuous on $\Cone(s_0,k_0) \setminus \{0\}$, and, in particular, on $V$.
Hence,
 \[
 A \Cone(s_0,k_0) \cap S^{q-1} = F \big(\Cone(s_0,k_0) \setminus \{0\}\big)
 = F(V).
 \]
By duality, inclusion~\eqref{eq::convexity} can be derived from the fact that
the supremum of any linear functional over the left side of \eqref{eq::convexity} does not exceed the supremum over the right side of it.
By the equality above, it is enough to show that for any $\theta \in S^{q-1}$,
there exists $z' \in \R^p \setminus \{0\}$ such that
$|\supp(z')| \leq d$ and $F(z')$ is well defined, which satisfies
\ben
\label{eq::maurey-z-prime}
\max_{v \in V} \ip{F(v), \theta}
& \leq & (1 - \delta)^{-1} \ip{F(z'), \theta}.
\een
For a given $\theta$, we construct a $d$-sparse vector $z'$ which satisfies~\eqref{eq::maurey-z-prime}. Let
$$z := \arg\max_{v \in V} \ip{F(v), \theta}.$$
By definition of $V$ there exists $I \subset \{1, \ldots, p\}$ such that $\size{I} = s_0$, and
for some $\ve_j \in \{1, -1\}$,
\ben
\label{eq::vecz}
z = z_{I} + \norm{z_I}_1 k_0 \sum_{j\in I^c} \alpha_j \ve_j e_j, \text{ where } \;
\alpha_j \in [0,  1], \sum_{j \in I^c} \alpha_j \leq 1,
\text{  and } 1 \geq \twonorm{z_I} \geq \inv{\sqrt{k_0 +1}}.
\een
Note if $\alpha_i = 1$ for some $i \in I^c$, then $z$ is a sparse vector itself, and
we can set $z' = z$ in order for~\eqref{eq::maurey-z-prime} to hold.
We proceed
assuming $\alpha_i \in [0, 1)$ for all $i \in I^c$ in~\eqref{eq::vecz} from now on,
in which case, we construct a required sparse vector $z'$ via Lemma~\ref{lemma::maurey}.
To satisfy the assumptions of this lemma, denote $e_{p+1} = \vec{0}$, $\ve_{p+1} = 1$ and set
$$\alpha_{p+1}  = 1 - \sum_{j\in I^c} \alpha_j, \; \text{ hence } \;
\alpha_{p+1} \in [0,  1].$$
Let
\bens
y := A z_{I^c} =
\norm{z_I}_1 k_0 \sum_{j\in I^c} \alpha_j \ve_j A e_j
= \norm{z_I}_1 k_0 \sum_{j\in I^c \cup \{p+1\}} \alpha_j \ve_j A e_j
\eens
and denote $\mathcal{M} := \{j \in I^c \cup \{p+1\}:  \alpha_j >0\}$.
Let $\ve >0$ be specified later.
Applying Lemma~\ref{lemma::maurey} with vectors
$u_j = k_0 \norm{z_I}_1 \ve_j A e_j$ for $j \in \mathcal{M}$,
construct a set $J' \subset \M$ satisfying
\ben
\label{eq::sample-m-I}
|J'| \leq m := \frac{4 \max_{j \in I^c}
k_0^2 \norm{z_I}_1^2 \twonorm{ A e_j}^2}{\ve^2}
\leq \frac{ 4 k_0^2 s_0 \max_{j \in I^c} \twonorm{A e_j}^2}{\ve^2}
\een
and a vector
$$y' = k_0 \norm{z_I}_1 \sum_{j\in J'} \beta_j \ve_j A e_j
\; \text{where for } \;  J' \subset \M, \beta_j \in [0, 1] \text{ and } \sum_{j \in J'} \beta_j = 1
$$
such that $\twonorm{y' -y} \leq \ve$.

Set $u := k_0 \norm{z_I}_1 \sum_{j\in J'} \beta_j \ve_j e_j$ and let
\ben
\nonumber
z' & = & z_I + u.
\een
By construction, $A z' \in AE_J$, where $J := (I \cup J')\cap\{1, \ldots, p\}$
and
\ben
\label{eq::define=J}
|J| \leq |I| + |J'| \leq s_0 + m.
\een
Furthermore, we have
\[
\twonorm{Az - A z'}  =  \twonorm{A(z_{I^c} - u)}=  \twonorm{y - y'} \leq \ve
\]
For $\{\beta_j, j \in J'\}$ as above, we extend it to $\{\beta_j, j \in I^c \cup \{p+1\}\}$
setting $\beta_j = 0$ for all $j \in I^c \cup \{p+1\} \setminus J'$ and write
\[
z'  =  z_{I} + k_0 \norm{z_I}_1  \sum_{j\in I^c \cup \{p+1\}} \beta_j \ve_j e_j\; \;
\text{where } \; \beta_j \in [0, 1] \;  \text{ and } \; \sum_{j \in I^c \cup \{p+1\}} \beta_j = 1.
\]
If $z' = z$, we are done.
Otherwise, for some $\lambda$ to be specified, consider the vector
\[
z + \lambda (z'  - z)  =
 z_{I} + k_0 \norm{z_I}_1  \sum_{j\in I^c \cup \{p+1\}}
\left[ (1 -\lambda) \alpha_j + \lambda \beta_j \right]\ve_j e_j.
\]
We have
$\sum_{j\in I^c \cup \{p+1\}}
\left[ (1 -\lambda) \alpha_j + \lambda \beta_j \right] = 1$
and
\[
\exists  \; \delta_0 > 0 \; \text{ s. t. } \;
\forall j \in I^c \cup \{p+1\}, \; \; \;
(1 -\lambda) \alpha_j + \lambda \beta_j \in [0, 1]  \; \text{if } |\lambda| < \delta_0.
\]
To see this, we note that
\bit
\item
This condition holds by continuity for all $j$ such that $\alpha_i \in (0, 1)$.
\item
If $\alpha_j = 0$ for some $j$, then $\beta_j = 0$ by construction.
\eit
Thus $\sum_{j\in I^c}
\left[ (1 -\lambda) \alpha_j + \lambda \beta_j \right] \leq 1$
and $z + \lambda (z'  - z) =
 z_{I} + k_0 \norm{z_I}_1  \sum_{j\in I^c}
\left[ (1 -\lambda) \alpha_j + \lambda \beta_j \right]\ve_j e_j \in V$
whenever $|\lambda| < \delta_0$.

Consider now a function $\phi: (-\delta_0, \delta_0) \to \R$,
\[
\phi(\lambda) := \ip{F(z + \lambda(z' -z)), \theta}
= \frac{\ip{Az + \lambda(Az' -Az), \theta}}{\twonorm{Az + \lambda(Az' -Az)}}
\]
Since $z$  maximizes $\ip{F(v), \theta}$ for all $v \in V$,
$\phi(\lambda)$ attains the local maximum at $0$.
Then by Lemma~\ref{lemma::local-maxima}, we have
\bens
\frac{\ip{Az', \theta}}{\ip{Az, \theta}} =
\frac{\ip{Az + (Az' -Az), \theta}}{\ip{Az, \theta}}
\geq 1 -\frac{\twonorm{(Az' -Az)}}{\twonorm{Az}} =
\frac{\twonorm{Az} - \twonorm{(Az' -Az)}}{\twonorm{Az}}
\eens
hence
\bens
\frac{\ip{F(z'), \theta}}{\ip{F(z), \theta}} &  = &
\frac{\ip{{Az'}/{\twonorm{A z'}}, \theta}}{\ip{{Az}/{\twonorm{Az}, \theta}}}
 =
\frac{\twonorm{A z}}{\twonorm{Az'}} \times \frac{\ip{Az', \theta}}{\ip{Az, \theta}} \\
& \geq &
\frac{\twonorm{A z}}{\twonorm{Az} + \twonorm{(Az' -Az)} } \times \frac{\twonorm{Az} - \twonorm{(Az' -Az)}}{\twonorm{Az}} \\
& = &
\frac{\twonorm{Az} - \twonorm{(Az' -Az)} }{\twonorm{Az} + \twonorm{(Az' -Az)}} \\
& = &
\frac{\twonorm{Az} - \ve}{\twonorm{Az} + \ve} = 1 - \frac{2\ve}{\twonorm{Az} + \ve}.
\eens
By definition, $z \in \W(s_0, k_0)$.
Hence we apply $\RE(k_0, s_0, A)$ condition and~\eqref{eq::vecz} to obtain
\[
\twonorm{A z} \geq \frac{\twonorm{z_I}}{K(s_0, k_0, A)} \geq \inv{\sqrt{1 + k_0}K(s_0, k_0, A)}.
\]
Now we can set $\ve = \frac{\delta}{2\sqrt{1 + k_0}K(s_0, k_0, A)}$
which yields
\ben
\label{eq::embedding-final}
\frac{\ip{F(z'), \theta}}{\ip{F(z), \theta}}
 \geq 1 -  \delta
\een
and thus~\eqref{eq::maurey-z-prime} holds.
Finally, by~\eqref{eq::sample-m-I}, we have
$$m \leq s_0
\max_{j \in I^c} \twonorm{A e_j}^2
\left(\frac{16 K^2(s_0, k_0, A) k_0^2 (k_0 + 1)}{\delta^2}\right)$$
and hence the inclusion~\eqref{eq::convexity} holds in view of~\eqref{eq::define=J} and~\eqref{eq::embedding-final}.
\end{proof}

\subsection{Proof of the reduction principle}

To  prove the restricted isomorphism condition \eqref{eq::sparse-isometry-II}, we apply Lemma~\ref{lemma::sparse-approx} with $k_0$ being replaced by $3 k_0$. The upper bound in \eqref{eq::sparse-isometry-II} follows immediately from the lemma. To prove the lower bound, we consider a vector $x \in \Cone(s_0,k_0)$ as an endpoint of an interval, whose midpoint is a sparse vector from the same cone. Then the other endpoint of the interval will be contained in the larger cone $\Cone(s_0,3k_0)$. Comparison between the upper estimate for the norm of the image of this endpoint with the lower estimate for the midpoint will yield the required lower estimate for the point $x$.

\begin{proofof}{Theorem~\ref{thm::isometry}}
Let $v \in \W(s_0, 3k_0) \setminus \{0\}$, 
and so $\twonorm{Av}>0$ by $\RE(s_0, 3k_0, A)$ condition. 
Let $d(3k_0, A)$ be defined as in \eqref{eq::definition-d(k_0,A)}. 
As in the proof of Lemma ~\ref{lemma::sparse-approx}, we 
may assume that $d(3k_0, A)<p$.
By Lemma~\ref{lemma::sparse-approx},
applied with $k_0$ replaced with $3k_0$, we have
\bens
\nonumber
& & \frac{A v}{\twonorm{A v}} \in A \Big(\W(s_0,3k_0)\Big) \cap S^{q-1}
\subset (1- \delta)^{-1}
\conv\left(\bigcup_{|J| = d(3k_0,A)} AE_J \cap S^{q-1}\right) \; \; \\
\label{eq::convex-hull}
& & \text{ and } \;
\twonorm{\frac{\tilde\Psi A v}{\twonorm{A v}}}
\leq \inv{1- \delta}
\max_{u \in \conv\left(A E \cap S^{q-1}\right)} \twonorm{\tilde\Psi u}
 = \inv{1 - \delta} \max_{u \in A E \cap S^{q-1}} \twonorm{\tilde\Psi u}.
\eens
The last equality holds, since the maximum of $\shtwonorm{\tilde\Psi u}$
occurs at an extreme point of the set $\conv(A E \cap S^{q-1})$,
because of convexity of the function $f(x)= \shtwonorm{\tilde\Psi x}$.
Hence, by \eqref{eq::sparse-isometry}
\ben
\label{eq::upper-cone-bound}
\forall x \in A \Big(\W(s_0,3k_0)\Big) \cap S^{q-1},
\quad
\norm{\tilde\Psi x}_2 \le (1+\delta)(1-\delta)^{-1} \leq 1 + 3 \delta
\een
where the last inequality is satisfied once $\delta < 1/3$,
which proves the upper estimate in~\eqref{eq::sparse-isometry-II}.

We have to prove the opposite inequality.
  Let $x=x_I+x_{I^c} \in  \W(s_0, k_0) \cap S^{p-1}$, where the set $I$ contains
the locations of the $s_0$
largest coefficients of $x$ in absolute values.
We have
\ben
\label{eq::vecx}
x = x_{I} + \norm{x_{I^c}}_1 \sum_{j\in I^c} \frac{|x_j|}{\norm{x_{I^c}}_1} \sign(x_j) e_j, \text{ where } \;
\; 1 \geq \twonorm{x_I} \geq \inv{\sqrt{k_0 +1}} \text{ by }~\eqref{eq::cone-top-norm}
\een
Let $\ve >0$ be specified later.
We now construct a $d(3k_0,A)$-sparse vector $y = x_I + u \in \W(s_0, k_0)$, where
$u$ is supported on $I^c$ which satisfies
\ben
\label{eq::maurey-y-x}
\norm{u}_1 = \norm{y_{I^c}}_1 = \norm{x_{I^c}}_1 \; \;
\text{ and } \;  \twonorm{A x - A y} = \norm{A (x_{I^c}- y_{I^c})}_2 \le \ve
\een
To do so, set
$$w := A x_{I^c} =  \norm{x_{I^c}}_1 \sum_{j\in I^c} \frac{|x_j|}{\norm{x_{I^c}}_1} \sign(x_j) A e_j.$$
Let $\mathcal{M} := \{j \in I^c:  x_j \not=0 \}$. Applying Lemma~\ref{lemma::maurey} with vectors
$u_j = \norm{x_{I^c}}_1 \sign(x_j) A e_j$ for $j \in \mathcal{M}$,
construct a set $J' \subset \M$ satisfying
\ben
\label{eq::sample-m-final}
|J'| \leq m := \frac{4 \max_{j \in \M}
\norm{x_{I^c}}_1^2 \twonorm{ A e_j}^2}{\ve^2}
\leq \frac{ 4 k_0^2 s_0 \max_{j \in \M} \twonorm{A e_j}^2}{\ve^2}
\een
and a vector
$$w' = \norm{x_{I^c}}_1 \sum_{j\in J'} \beta_j \sign(x_j) A e_j,
\; \text{where for } \;  J' \subset \M, \beta_j \in [0, 1] \text{ and } \sum_{j \in J'} \beta_j = 1
$$
such that $\twonorm{Ax-Ay}=\twonorm{w' -w} \leq \ve$. Set $u := \norm{x_{I^c}}_1 \sum_{j\in J'} \beta_j \sign(x_j) e_j$ and let
\[
y  =  x_I + u =  x_{I} + \norm{x_{I^c}}_1  \sum_{j\in J'} \beta_j \sign(x_j) e_j\; \;
\text{where } \; \beta_j \in [0, 1] \;  \text{ and } \; \sum_{j \in J'} \beta_j = 1.
\]
By construction, $y \in \W(s_0, k_0) \cap E_J$, where $J := I \cup J'$
and
\ben
\label{eq::define=d=final}
|J| = |I| + |J'| \leq s_0 + m.
\een
This, in particular, implies that $\twonorm{Ay}>0$.
Assume that $\ve$ is chosen so that $s_0 + m \leq d(3k_0,A)$, and so by
~\eqref{eq::sparse-isometry}
\bens
  \norm{\frac{\tilde\Psi A y}{\norm{Ay}_2}}_2 \ge 1-\d.
\eens
Set
\ben
\label{eq::define=v}
   v=x_I+2y_{I^c}-x_{I^c}=y+(y_{I^c}-x_{I^c}).
  \een
Then~\eqref{eq::maurey-y-x} implies
\ben
\label{eq::Av-approx}
\norm{Av}_2 \le \norm{Ay}_2 + \norm{A(y_{I^c}-x_{I^c})} \leq
\norm{Ay}_2 +   \ve,
\een
and $v \in   \W(s_0, 3k_0)$ as
$$\norm{v_{I^c}}_1 \leq 2\norm{y_{I^c}}_1 +\norm{ x_{I^c}}_1 = 3 \norm{x_{I^c}}_1
\le 3k_0 \norm{x_I}_1 =  3k_0 \norm{v_I}_1$$
where we use the fact that $\norm{x_{I^c}}_1 =  \norm{y_{I^c}}_1$.
Hence, by the upper estimate~\eqref{eq::upper-cone-bound}, we have
\ben
\label{eq::upp-cone}
\norm{\frac{\tilde\Psi A v}{\norm{Av}_2}}_2 \le (1+ \delta)(1-\delta)^{-1}
\een
  Since $y=\frac{1}{2}(x+v)$, where $y_I = x_I$, we have by the lower bound in~\eqref{eq::sparse-isometry}
 and the triangle inequality, \\
\bens
1-\delta
  &\le & \norm{\frac{\tilde\Psi A y}{\norm{Ay}_2}}_2
  \le \frac{1}{2} \left(\norm{\frac{\tilde\Psi A x}{\norm{Ay}_2}}_2
    +\norm{\frac{\tilde\Psi A v}{\norm{Ay}_2}}_2 \right)       \\
  &\le & \frac{1}{2} \left(\norm{\frac{\tilde\Psi A x}{\norm{Ax}_2}}_2
    +\norm{\frac{\tilde\Psi A v}{\norm{Av}_2}}_2  \right)
        \cdot \frac{\norm{Ay}_2+\ve}{\norm{Ay}_2}
      \\
  &\le & \frac{1}{2} \left(\norm{\frac{\tilde\Psi A x}{\norm{Ax}_2}}_2
    + \frac{1+ \delta}{1-\delta} \right) \cdot \big(1+\delta/6\big)
\eens
where in the second line, we apply~\eqref{eq::Av-approx} and~\eqref{eq::maurey-y-x},
and in the third line, \eqref{eq::upp-cone}.
By the $\RE(s_0, k_0, A)$ condition and~\eqref{eq::vecx}
we have
$$\norm{Ay}_2 \ge \frac{ \twonorm{y_I} }{K(s_0, k_0, A)}
= \frac{ \twonorm{x_I} }{K(s_0, k_0, A)} \geq
\frac{1 }{K(s_0, k_0, A) \cdot \sqrt{k_0 + 1}}.$$
Set
$$\ve = \frac{\delta}{6\sqrt{1 + k_0}K(s_0, k_0, A)} \; \text{ so that } \;
\frac{\norm{Ay}_2+\ve}{\norm{Ay}_2}
\leq (1+\delta/6\big).
$$
Then for $\delta < 1/5$
\bens
\norm{\frac{\tilde\Psi A x}{\norm{Ax}_2}}_2 \geq 2 \frac{1-\delta}{1+\delta/6} - (1+ \delta)(1-\delta)^{-1}
\geq 1- 5 \delta.
\eens
This verifies the lower estimate. It remains to check the bound for the cardinality of $J$.
By~\eqref{eq::sample-m-final} and~\eqref{eq::define=d=final}, we have for $k_0 > 0$,
\bens
|J| & \leq & s_0 + m \leq
s_0 +  s_0 \max_{j \in \M} \twonorm{A e_j}^2 \left(\frac{16 K^2(s_0, k_0, a) (3k_0)^2 (k_0 + 1)}{\delta^2}\right) < d(3k_0,A)
\eens
as desired. This completes the proof of
Theorem~\ref{thm::isometry}.


\end{proofof}
\begin{remark}
\label{remark::RE}
 Let $\e>0$. Instead of $v$ defined in \eqref{eq::define=v}, one can consider the vector
 \[
   v_{\e}=x_I+y-\e(x-y) \in \Cone \big(s_0,(1+\e)k_0 \big).
 \]
 Then replacing $v$ by $v_\e$ throughout the proof, we can establish Theorem \ref{thm::isometry} under the assumption $\RE(s_0, (1+\e)k_0, A)$ instead of $\RE(s_0, 3k_0, A)$, if we increase the dimension $d(3k_0)$ by a factor depending on $\e$.
\end{remark}

\section{Subgaussian random design}
\label{sec:subgaussian-proof}
Theorem~\ref{thm:subgaussian-T-intro}
can be reformulated as an almost isometry condition for the
matrix $X=\Psi A$ acting on the set $\Cone(s_0,k_0)$.
Recall that
\[
  d(3k_0, A)  = s_0 +  s_0 \max_j \twonorm{A e_j}^2 \left(\frac{16 K^2(s_0, 3k_0, A) (3k_0)^2 (3k_0 + 1)}{\delta^2}\right).
\]
\begin{theorem}
\label{thm:subgaussian-T}
Set $0< \d < 1$, $0< s_0 < p$, and $k_0 > 0$.
Let $A$ be a $q \times p$ matrix satisfying $\RE(s_0, 3k_0, A)$ condition
as in Definition~\ref{def:memory}.
Let $m = \min(d(3k_0, A),p) < p$. 
Let $\Psi$ be an $n \times q$  matrix whose rows are
independent isotropic $\psi_2$ random vectors in $\R^q$ with constant $\alpha$.
Assume that the sample size satisfies
\ben
\label{eq::UpsilonSampleBound-thm}
n \geq \frac{2000 m \alpha^4}{\d^2} \log \left(\frac{60 e p}{m \d}\right).
\een
Then with probability at least $1- 2 \exp(\d^2 n/2000 \alpha^4)$,
 for all $\upsilon \in \Cone(s_0, k_0)$ such that $\upsilon \not= 0$,
\begin{eqnarray}
\label{eq::phi-bound}
1 - \d & \leq &
\frac{1}{\sqrt{n}}\frac{\twonorm{\Psi A \upsilon}}{\twonorm{A \upsilon}}\;
\leq \; 1 + \delta.
\end{eqnarray}
\end{theorem}

Theorem~\ref{thm:subgaussian-T-intro}
follows immediately from Theorem~\ref{thm:subgaussian-T}.
Indeed, by~\eqref{eq::phi-bound},  for all $u \in \W(s_0,k_0)$ such that $u \not=0$,
\bens
\norm{\frac{1}{\sqrt{n}}\Psi A  u}_2 \ge (1-\d) \twonorm{A u} \ge
  (1-\d) \frac{\norm{u_{T_0}}_2}{K(s_0, k_0,A)} > 0.
\eens

To derive Theorem \ref{thm:subgaussian-T} from Theorem \ref{thm::isometry}
we need a lower estimate for the norm of the image of a sparse vector.
Such estimate relies on the standard $\e$-net argument similarly
to~\citet[Section 3]{MPT08}.


\begin{theorem}
\label{thm:subgaussian-T-II}
Set $0< \delta < 1$.
Let $A$ be a $q \times p$ matrix, and
let $\Psi$ be an $n \times q$, matrix  whose rows are
independent isotropic $\psi_2$ random vectors in $\R^q$ with constant $\alpha$.
For $m \leq p$, assume that
\ben
\label{eq::sample-size-gen-II}
n \geq \frac{80 m \alpha^4}{\tau^2} \log \left(\frac{12 e p}{m \tau}\right).
\een
Then with probability at least $1- 2 \exp(-\tau^2 n/80 \alpha^4)$,
for all $m$-sparse vectors $u$ in $\R^p$,
\begin{eqnarray}
\label{eq::phi-bound-2}
& & (1- \tau)\twonorm{A u} \; \le \; \inv{\sqrt{n}} \twonorm{\Psi A u} \; \leq \; (1 + \tau)\twonorm{A u}.
\end{eqnarray}
\end{theorem}
We note that Theorem~\ref{thm:subgaussian-T-II}
does not require the RE condition to hold.
No particular upper bound on $\rho_{\max}(m,A)$ is imposed here either.

We now state a large deviation bounds for $m$-sparse eigenvalues
$\rho_{\min}(m, \tilde{X})$ and $\rho_{\max}(m,\tilde{X})$
for random design $\tilde{X} = n^{-1/2}\Psi A$ which
follows from Theorem~\ref{thm:subgaussian-T-II} directly.
\begin{corollary}
Under conditions in Theorem~\ref{thm:subgaussian-T-II}, 
we have with probability at least $1- 2 \exp(-\tau^2 n/80 \alpha^4)$,
\begin{eqnarray}
\label{eq::phi-max-bound-II}
 (1- \tau) \sqrt{\rho_{\min}(m, A)} \leq \sqrt{\rho_{\min}(m, \tilde{X})} \leq
\sqrt{\rho_{\max}(m, \tilde{X})} \leq (1+ \tau) \sqrt{\rho_{\max}(m,A)}.
\end{eqnarray}
\end{corollary}

\subsection{Proof of Theorem~\ref{thm:subgaussian-T}}
For $n$ as bounded in~\eqref{eq::UpsilonSampleBound-thm},
where $m=\min(d(3k_0, A), p)$,
we have~\eqref{eq::sample-size-gen-II} holds with $\tau = \delta/5$.
Then by Theorem~\ref{thm:subgaussian-T-II}, we have
with probability at least  $1 -  2\exp\left(-n \delta^2/(2000 \alpha^4)\right)$,
\[
\forall \text{$m$-sparse vectors $u$,} \quad
 \left(1- \frac{\delta}{5} \right)\twonorm{A u}  \; \le \; \frac{1}{\sqrt{n}} \twonorm{\tilde\Psi A u} \; \leq \;
  \left( 1 + \frac{\delta}{5} \right) \twonorm{A u}.
 \]
The proof finishes by application of Theorem~\ref{thm::isometry}.

\subsection{Proof of Theorem~\ref{thm:subgaussian-T-II}}
We start with a definition.

\begin{definition}
Given a subset $U \subset \R^p$ and a number $\ve > 0$,
an $\ve$-net $\Pi$ of $U$ with respect to the Euclidean metric
is a subset of points of $U$ such that $\ve$-balls
centered at $\Pi$ covers $U$:
$$U \subset \bigcup_{x \in \Pi} (x + \ve \Ball_2^p),$$
where $A + B := \{a + b: a \in A, b \in B \}$ is the Minkowski sum of
the sets $A$ and $B$.
The covering number $\Net(U, \ve)$ is the smallest cardinality of an
$\ve$-net of $U$.
\end{definition}

The proof of Theorem \ref{thm:subgaussian-T-II} uses two well-known results.
The first one is the {\em volumetric estimate}; see e.g.~\cite{MS86}.
\begin{lemma}
\label{eq::Pi-cover-numbers}
Given $m \geq 1$ and $\ve >0$. There exists an $\ve$-net
$\Pi \subset B_2^m$ of $B_2^m$ with respect to the Euclidean metric such
that $B_2^m \subset (1- \ve)^{-1} \conv \Pi$ and  $|\Pi| \leq (1+2/\ve)^m$.
Similarly, there exists an $\ve$-net of the sphere
$S^{m-1}$, $\Pi' \subset S^{m-1}$ such that $|\Pi'| \leq (1+2/\ve)^m$.
\end{lemma}

 The second lemma with a worse constant can be derived
 from Bernstein's inequality for subexponential random variables.
 Since we are interested in the numerical value of the constant, we
 provide a proof below.
 \begin{lemma}
\label{lemma::bernstein}
  Let $Y_1 \etc Y_n$ be independent  random variables such that $\E Y_j^2=1$ and
   $\norm{Y_j}_{\psi_2} \le \a$ for all $j=1 \etc n$. Then for any
  $\theta \in (0,1)$
  \[
    \P \left( \left| \frac{1}{n} \sum_{j=1}^n Y_j^2-1 \right|
    >\theta \right)
    \le 2 \exp \left( - \frac{\theta^2 n}{10 \a^4} \right).
  \]
 \end{lemma}

For a set $J \subset \{1, \ldots, p\}$, denote
$E_J = \span\{e_j: j \in J\}$, and set $F_J=A E_J$.
For each subset $F_J \cap S^{q-1}$,  construct an $\ve$-net $\Pi_J$,
which satisfies
$$\Pi_J \subset  F_J \cap S^{q-1} \; \; \text{and } \; \; |\Pi_J| \leq (1+2/\ve)^m.$$
The existence of such $\Pi_J$ is guaranteed by Lemma~\ref{eq::Pi-cover-numbers}.
 If
$$\Pi = \bigcup_{|J| = m} \Pi_J,$$
then the previous estimate implies
$$\size{\Pi} = (3/\ve)^m {p \choose m} \leq
\left(\frac{3e p}{m \ve}\right)^m =
\exp\left(m \log \left(\frac{3e p}{m \ve}\right)\right)
$$
For $y \in S^{q-1} \cap F_J \subset F$,
let $\pi(y)$ be one of the closest point in the $\ve$-cover
$\Pi_J$.
Then
\bens
\frac{y - \pi(y)}{\twonorm{y - \pi(y)}} \in F_J \cap S^{q-1}
\; \; \text{where} \; \; \twonorm{y - \pi(y)} \leq \ve.
\eens
Denote by
$\Psi_1, \ldots, \Psi_n$ the rows of the matrix $\Psi$, and
set $\Gamma=n^{-1/2}\Psi$.  Let $x \in S^{q-1}$.
 Applying Lemma~\ref{lemma::bernstein}
to the  random variables $\ip{\Psi_1, x}^2, \ldots, \ip{\Psi_n, x}^2$,
we have that for every $\theta <1$
\ben
\prob{\size{\twonorm{\Gamma x}^2 - 1} > \theta}
& = &
\label{eq::Bernstein}
\prob{\size{\inv{n} \sum_{i=1}^n \ip{\Psi_i, x}^2 - 1} > \theta} \leq
2 \exp\left(-\frac{n \theta^2}{10 \alpha^4}\right).
\een
For
\bens
n \geq \frac{20 m \alpha^4}{\theta^2} \log \left(\frac{3e p}{m \ve}\right),
\eens
 the union bound implies
\bens
\prob{\exists x \in \Pi \text{ s. t. }
\size{\twonorm{\Gamma x}^2 - 1} > \theta}  \leq
2 \size{\Pi} \exp\left(-  \frac{n \theta^2}{10 \alpha^4}\right)
\leq
2 \exp\left(-  \frac{n \theta^2}{20 \alpha^4}\right)
\eens
Then for all $y_0 \in \Pi$
\bens
& & 1 - \theta \leq \twonorm{\Gamma y_0}^2 \leq 1 + \theta
\; \; \text{and so } \; \;  \\
& & 1 - \theta \leq \twonorm{\Gamma y_0} \leq 1 + \frac{\theta}{2}
\eens
with probability at least
$1 - 2 \exp\left(-  \frac{n \theta^2}{20 \alpha^4}\right)$,
The bound
over the entire $S^{q-1} \cap F_J$ is obtained by approximation.
We have
\ben
\label{eq::triangle-net}
\twonorm{\Gamma \pi(y)} - \twonorm{\Gamma (y- \pi(y))} \leq
\twonorm{\Gamma y} \leq
\twonorm{\Gamma \pi(y)}
+\twonorm{\Gamma (y- \pi(y))}
\een
Define
$$
\norm{\Gamma }_{2, F_J} := \sup_{y \in S^{q-1} \cap F_J}
\twonorm{\Gamma y}.$$
The RHS of~\eqref{eq::triangle-net} is upper bounded by
$1 + \frac{\theta}{2} + \ve  \norm{\Gamma }_{2, F_J}$. By
taking the supremum over all $y \in S^{q-1} \cap F_J$, we have
\bens
\norm{\Gamma}_{2, F_J} \leq  1 + \frac{\theta}{2} +
\ve  \norm{\Gamma}_{2, F_J}
\; \; \text{and hence } \; \;
\norm{\Gamma}_{2, F_J} \leq \frac{1 + \theta/2}{1-\ve}.
\eens
The LHS of~\eqref{eq::triangle-net} is lower bounded by
$1 - \theta - \ve  \norm{\Gamma}_{2, F_J}$, and hence for all
$y \in S^{q-1} \cap F_J$
\bens
\twonorm{\Gamma y}
& \geq &  1 - \theta - \ve  \norm{\Gamma}_{2, F_J}  \geq
1 - \theta -  \frac{\ve(1 + \theta/2)}{1-\ve}
\eens
Putting these together, we have for all $y \in S^{q-1} \cap F_J$
\bens
1 - \theta -  \frac{\ve(1 + \theta/2)}{1-\ve} \leq
\twonorm{\Gamma y} \leq  \frac{1 + \theta/2}{1-\ve}
\eens
which holds for all sets $J$.
Thus for $\theta < 1/2$ and $\ve = \frac{\theta}{1+ 2\theta}$,
\bens
1 - 2\theta  < \twonorm{\Gamma y} < 1 + 2 \theta.
\eens
For any $m$-sparse vector $u \in S^{p-1}$
\bens
\frac{A u}{\twonorm{A u}} \in F_J
\; \; \text{ for } J=\supp(u),\; \;
\eens
and so
\[
(1 - 2\theta )
\twonorm{A u} \leq
\twonorm{\Gamma A u } \leq (1 + 2\theta) \twonorm{A u}.
\]
Taking $\tau = \theta/2$ finishes the proof for  Theorem~\ref{thm:subgaussian-T-II}.

\subsection{Proof of Lemma~\ref{lemma::bernstein}}
\begin{proofof2}
 Note that  $\a \ge \norm{Y_1}_{\psi_2} \ge \norm{Y_1}_{2} = 1$.
  Using the elementary inequality $t^k \le k! s^k e^{t/s}$, which holds for all $t,s > 0$, we
  obtain
  \[
    |\E (Y_j^2-1)^k| \le \max (\E Y_j^{2k}, 1)
    \le \max ( k! \alpha^{2k} \cdot \E e^{Y_j^2/\a^2}, 1)
    \le 2 k! \a^{2k}
  \]
  for any $k \ge 2$. Since for any $j$ $\E Y_j^2=1$, for any $\t \in \R$
  with $|\t| \a^2 <1$
  \begin{multline*}
   \E \exp \left[ \t (Y_j^2-1) \right]
   \le 1+ \sum_{k=2} \frac{1}{k!} |\t|^k \cdot  |\E (Y_j^2-1)^k|
   \le 1+ \sum_{k=2} |\t|^k \cdot 2 \a^{2k}  \\
   \le 1+ \frac{2 \t^2 \a^4}{1-|\t| \a^2}
   \le \exp \left(\frac{2 \t^2 \a^4}{1-|\t| \a^2} \right).
  \end{multline*}
  By Markov's inequality, for $\t \in (0, \a^{-2})$
  \begin{multline*}
    \P \left(  \frac{1}{n} \sum_{j=1}^n Y_j^2-1
    >\theta \right)
    \le \E \exp \left(\t   \sum_{j=1}^n (Y_j^2-1)
    -\t \theta n \right) \\
    = e^{-\t \theta n} \cdot \left(\E \exp \left[ \t (Y^2-1) \right]
    \right)^n
    \le \exp \left( -\t \theta n + \frac{2 \t^2 \a^4 n}{1-|\t| \a^2}
    \right).
  \end{multline*}
  Set $\t=\frac{\theta}{5 \a^4}$, so $\t \a^2 \le 1/5$.
  Then the previous inequality implies
  \[
    \P \left(  \frac{1}{n} \sum_{j=1}^n Y_j^2-1
    >\theta \right)
    \le \exp \left(- \frac{\theta^2 n}{10 \a^4} \right).
  \]
  Similarly, considering $\t<0$, we obtain
  \[
    \P \left( 1- \frac{1}{n} \sum_{j=1}^n Y_j^2
    >\theta \right)
    \le \exp \left(- \frac{\theta^2 n}{10 \a^4} \right).
  \]

 \end{proofof2}

\section{RE condition for random matrices with bounded entries}
\label{sec:bounded-proof}

We next consider the case of design matrix $X$ consisting of independent identically distributed rows
 with bounded entries. As in the previous section, we reformulate
 Theorem \ref{thm::RE-bounded-entries-intro} in the form of an almost
 isometry condition.
 
\begin{theorem}  \label{thm::RE-bounded-entries}
 Let $0<\d<1$and $0< s_0<p$. Let $Y \in \R^p$ be a random vector such that
 $\norm{Y}_{\infty} \le M$ a.s., and denote $\S= \E Y Y^T$. Let $X$ be an
$n \times p$ matrix,  whose rows $X_1 \etc X_n$ are independent copies of $Y$.
Let $\Sigma$ satisfy the $\RE(s_0, 3k_0, \Sigma^{1/2})$ condition
as in Definition~\ref{def:memory}.
Set
\[
  d=d(3k_0, \Sigma^{1/2})  = s_0 +  s_0 \max_j \twonorm{\Sigma^{1/2} e_j}^2 \left(\frac{16 K^2(s_0, 3k_0, \Sigma^{1/2}) (3k_0)^2 (3k_0 + 1)}{\delta^2}\right).
\]
Assume that $d \leq p$ and $\r=\r_{\min}(d,\S^{1/2}) >0$. 
If for some absolute constant $C$
  \[
    n \ge \frac{C M^2 d   \cdot \log p}{\r \d^2} \cdot
      \log^3 \left( \frac{C M^2 d \cdot \log p}{\r \d^2 }\right),
  \]
 then with probability at least
  $1- \exp \left( -  \d \r n/(6 M^2 d) \right)$ all
  vectors $u \in \Cone(s_0, k_0)$ satisfy
  \[
   ( 1-\d) \twonorm{u} \le \frac{\norm{X u}_2}{\sqrt{n}} \le (1+\d)  \twonorm{u}.
  \]
 \end{theorem}

 Similarly to Theorem~\ref{thm:subgaussian-T}, Theorem \ref{thm::RE-bounded-entries} can
 be derived from Theorem~\ref{thm::isometry},  and the corresponding bound for
 $d$-sparse vector, which is proved below.

 \begin{theorem} \label{thm::sparse-bounded-entries}
  Let $Y \in \R^p$ be a random vector such that $\norm{Y}_{\infty}
  \le M$ a.s., and denote $\S= \E Y Y^T$. Let $X$ be an $n \times p$ matrix,
  whose rows $X_1 \etc X_n$ are independent copies of $Y$.
Let $0< m \leq p$.  If $\r=\r_{\min}(m, \S^{1/2})>0$ and
  \begin{equation}  \label{bound on n}
    n \ge \frac{C M^2 m  \cdot \log p}{\r \d^2} \cdot
      \log^3 \left( \frac{C M^2 m   \cdot \log p}{\r \d^2 }\right),
  \end{equation}
  then with probability at least
  $1- 2 \exp \left( - \frac{  \e \r n}{6 M^2 m} \right)$ all $m$-sparse
  vectors $u$ satisfy
  \[
    1-\d
    \le \frac{1}{\sqrt{n}} \cdot \norm{\frac{X u}{\norm{\Sigma^{1/2} u}_2}}_2
    \le 1+\d.
  \]
 \end{theorem}

 To prove Theorem \ref{thm::sparse-bounded-entries} we consider random variables $Z_u=\norm{X u}_2/(\sqrt{n} \norm{\Sigma^{1/2} u}_2)-1$, and estimate the expectation of the supremum of  $Z_u$ over the set of sparse vectors using Dudley's entropy integral. The proof of this part closely follows \cite{RV08}, so we will only sketch it. To derive the large deviation estimate from  the bound on the expectation we use Talagrand's measure concentration theorem for empirical processes, which provides a sharper estimate, than the method used in \cite{RV08}.

\begin{proof}

 For $J \subset \{1 \etc p\}$, let $E_J$ be the coordinate subspace spanned by the vectors $e_j, \ j \in J$.
 Set
 \[
   F = \bigcup_{|J|=m} \S^{1/2} E_J \cap S^{p-1}.
 \]
 Denote $\Psi=\S^{-1/2} X$ so $\E \Psi \Psi^T =\id$,
 and let $\Psi_1 \etc \Psi_n$ be independent copies of $\Psi$.
 It is enough to show that
 with probability at least $1- \exp \left( - \frac{  \e \r n}{6 M^2 m} \right)$
 for any $y \in
 F$
 \[
   \left| 1- \frac{1}{n} \sum_{j=1}^n \pr{\Psi_j}{y}^2 \right| \le
   \d.
 \]
 To this end we estimate
 \[
   \D:=\E \sup_{y \in F}\left| 1- \frac{1}{n} \sum_{j=1}^n \pr{\Psi_j}{y}^2
  \right|.
 \]

 The standard symmetrization argument implies that
 \[
  \E \sup_{y \in F}\left| 1- \frac{1}{n} \sum_{j=1}^n \pr{\Psi_j}{y}^2
  \right|
  \le \frac{2}{n} \E \sup_{y \in F} \left| \sum_{j=1}^n \e_j
        \pr{\Psi_j}{y}^2 \right|,
 \]
 where $\e_1 \etc \e_n$ are independent Bernoulli random variables
 taking values $\pm 1$ with probability $1/2$.
 The estimate of the last quantity is based on the following Lemma,
 which is similar to Lemma 3.6 \cite{RV08}.
 \begin{lemma}  \label{l: bernoulli}
  Let $F$ be as above, and let $\psi_1 \etc \psi_n \in \R^p$.
  Set
  \[
   Q=\max_{j=1 \etc n} \norm{\S^{1/2} \psi_j}_{\infty}.
  \]
  Then
\bens
    \E \sup_{y \in F} \left| \sum_{j=1}^n \e_j  \pr{\psi_j}{y}^2
    \right|
       \le
     \sqrt{\frac{C m Q^2 \cdot \log n \cdot \log p}{\r }} \cdot
      \log \left( \frac{C m Q^2}{\r }\right)
    \ \cdot \
     \sup_{y \in F} \left( \sum_{j=1}^n \pr{\psi_j}{y}^2
    \right)^{1/2}.
 \eens
 \end{lemma}
 Assuming Lemma \ref{l: bernoulli}, we finish the proof of the
 Theorem. First, note that by the definition of $\Psi_j$,
 \[
   \max_{j=1 \etc n} \norm{\S^{1/2} \Psi_j}_{\infty} \le M \
   \text{a.s.}
 \]
 Hence, conditioning on $\Psi_1 \etc \Psi_n$ and applying Lemma
 \ref{l: bernoulli}, we obtain
  \[
    \D
    \le \frac{2}{n} \cdot  \sqrt{\frac{C m M^2 \cdot \log n \cdot \log p}{\r }} \cdot
      \log \left( \frac{C m M^2}{\r }\right)
    \ \cdot \E \sup_{y \in F}  \left(\sum_{j=1}^n \pr{\Psi_j}{y}^2
    \right)^{1/2},
  \]
  and by Cauchy--Schwartz inequality,
  \[
   \E \sup_{y \in F}  \left( \sum_{j=1}^n \pr{\Psi_j}{y}^2  \right)^{1/2}
      \leq  \left(  \E \sup_{y \in F}\sum_{j=1}^n \pr{\Psi_j}{y}^2 \right)^{1/2},
  \]
 so
  \[
    \D
    \le  \frac{2}{\sqrt{n}} \cdot  \sqrt{\frac{C m M^2 \cdot \log n \cdot \log p}{\r }} \cdot
      \log \left( \frac{C m M^2}{\r }\right)  \ \cdot \ \left( \D+1 \right)^{1/2}.
  \]
  If $n$ satisfies \eqref{bound on n}, then
  \[
    \D    \le  \d
    \ \cdot \ \left( \D+1
    \right)^{1/2},   \text{and thus } \; \D \le 2 \d.
  \]

  For $y \in F$ define a random variable $f(y)= \pr{\Psi}{y}^2 -1$.
  Then $|f(y)| \le \pr{X}{\S^{-1/2}y}^2 + 1 \le M^2 \r^{-1} m + 1 :=a$
  a.s., because $\S^{-1/2}y$ is an $m$-sparse vector, whose norm does not exceed
  $\r^{-1/2}$. Set
  \[
    Z=\sup_{y \in F}\sum_{j=1}^n f_j(y),
  \]
  where $f_1(y) \etc f_n(y)$ are independent copies of $f(y)$. The
  argument above shows that $\E Z \le 2 \d n$.
  Then
  Talagrand's concentration inequality for empirical processes
  \cite{L01} reads
  \[
   \P( Z \ge t ) \le \exp \left( - \frac{t }{6 a} \right)
   \le \exp \left( - \frac{t  \r}{6 M^2 m} \right)
  \]
  for all $t \ge 2 \E Z$. Setting $ t = 4 \d n$,
we have
  \[
   \P(\sup_{y \in F}\sum_{j=1}^n  \left( \pr{\Psi_j}{y}^2 -1\right) \ge 4\d n )    \le \exp \left( - \frac{4 \d n \r}{6 M^2 m} \right).
  \]
Similarly, considering random variables $g(y) = 1 - \pr{\Psi}{y}^2$,
we show that
  \[
   \P(\sup_{y \in F}\sum_{j=1}^n  \left(1- \pr{\Psi_j}{y}^2\right) \ge 4\d n )    \le \exp \left( - \frac{4 \d n \r}{6 M^2 m} \right),
  \]
 which completes the proof of the theorem.
\end{proof}

  It remains to prove Lemma \ref{l: bernoulli}. By Dudley's
  inequality
  \[
    \E \sup_{y \in F} \left| \sum_{j=1}^n \e_j  \pr{\psi_j}{y}^2
    \right|
    \le C \int_0^{\infty} \log^{1/2} N(F,d,u) \, du.
  \]
  Here $d$ is the natural metric of the related Gaussian process
  defined as
  \begin{eqnarray*}
    d(x,y)
    &=&\left[ \sum_{j=1}^n \left( \pr{\psi_j}{x}^2- \pr{\psi_j}{y}^2
                \right)^2 \right]^{1/2} \\
    &\le& \left[ \sum_{j=1}^n \left( \pr{\psi_j}{x}+ \pr{\psi_j}{y}
                \right)^2 \right]^{1/2} \cdot \max_{j=1 \etc n} |\pr{\psi_j}{x-y}|\\
    &\le& 2R \cdot \norm{x-y}_Y,
  \end{eqnarray*}
  where
  \[
   R= \sup_{y \in F} \left( \sum_{j=1}^n \pr{\psi_j}{y}^2   \right)^{1/2}, \quad
   \text{and }
   \norm{z}_Y=\max_{j=1 \etc n} |\pr{\psi_j}{z}|.
  \]
   The inclusion $\sqrt{m} B_1^p \supset   \bigcup_{|J|=m} E_J \cap S^{p-1}$ implies
 \[
   \sqrt{m} \S^{1/2} B_1^p \supset \S^{1/2} \conv ( \bigcup_{|J|=m} E_J \cap S^{p-1}) \supset \r^{1/2} F.
 \]
 Hence, for any $y \in F$
  \begin{equation}   \label{norm_Y}
    \norm{z}_Y \le \r^{-1/2} \sqrt{m} \max_{j=1 \etc n}
    \norm{\S^{1/2}\psi_j}_{\infty}= \r^{-1/2} \sqrt{m} Q.
  \end{equation}
  Replacing the metric $d$ with the norm $\norm{\cdot}_Y$, we obtain
  \[
    \E \sup_{y \in F} \left| \sum_{j=1}^n \e_j  \pr{\psi_j}{y}^2
    \right|
    \le C R \int_0^{\r^{-1/2} \sqrt{m} Q} \log^{1/2} N(F,\norm{\cdot}_Y,u) \, du.
  \]
  The upper limit of integration is greater or equal than the diameter of $F$ in the norm $\norm{\cdot}_Y$, so for $u > \r^{-1/2} \sqrt{m} Q$ the integrand is 0.
  Arguing as in Lemma 3.7 \cite{RV08}, we can show that
  \begin{equation}  \label{Maurey}
   N(F,\norm{\cdot}_Y,u) \le N(\r^{-1/2} \sqrt{m} \S^{1/2} B_1^p, \norm{\cdot}_Y,u)
   \le (2p)^l,
  \end{equation}
  where
  \[
    l=\frac{C \r^{-1}m \left( \max_{i=1 \etc p} \max_{j=1 \etc n} |\pr{\S^{1/2}
    e_i}{\psi_j}| \right)^2}{u^2} \cdot \log n
    = \frac{C m Q^2 \cdot \log n}{\r u^2}
  \]

  Also, since $F$ consists of the union
 $\binom{p}{m}$  Euclidean spheres, the inclusion \eqref{norm_Y} and  the volumetric
  estimate yield
  \begin{eqnarray} \label{volumetric}
  N(F,\norm{\cdot}_Y,u)
  &\le& \binom{p}{m} \cdot \left(1+ \frac{2\r^{-1/2} \sqrt{m} Q}{u} \right)^m
\; \le \; \left(\frac{e p}{m} \right)^m \cdot
      \left(1+ \frac{2\r^{-1/2} \sqrt{m} Q}{u} \right)^m.
  \end{eqnarray}
  Estimating the covering number of $F$ as in \eqref{Maurey} for $u
  \ge 1$, and as in \eqref{volumetric} for $0<u<1$, we obtain
  \begin{eqnarray*}
   & &\E \sup_{y \in F} \left| \sum_{j=1}^n \e_j  \pr{\psi_j}{y}^2
    \right| \\
   & \le
    &C R \int_0^1 \sqrt{m} \cdot \left( \log \left(\frac{e p}{m} \right)
        + \log \left(1+ \frac{2\r^{-1/2} \sqrt{m} Q}{u} \right) \right)^{1/2} \, du
        \\
   & &+ CR \int_1^{\r^{-1/2} \sqrt{m} Q} \sqrt{\frac{C m Q^2 \cdot \log n}{\r
    u^2}} \cdot \sqrt{\log 2 p} \, du \\
    &\le
    &CR \sqrt{\frac{ m Q^2 \cdot \log n \cdot \log p}{\r }} \cdot
      \log \left( \frac{C m Q^2}{\r }\right). \qed
  \end{eqnarray*}
  \begin{remark}
\label{remark::lower-bound}
  Note that unlike the case of a random matrix with subgaussian marginals,
  the estimate of Theorem \ref{thm::sparse-bounded-entries} contains the minimal sparse singular value $\r$. This is, however, necessary, as the following example shows.

  Let $m=2^l$, and assume that  $p=k \cdot m$, for some $k \in \N$. For $j=1 \etc k$ let $D_j$ be the $m \times m$ Walsh matrix. Let $A$ be a $p \times p$ block-diagonal matrix with blocks $D_1 \etc D_k$ on the diagonal, and let $Y \in \R^p$ be a random vector, whose values are the rows of the matrix $A$ taken with probabilities $1/p$. Then $\norm{Y}_{\infty}=1$ and $\E Y Y^T= (m/p) \cdot \id$, so $\rho=m/p$. Hence, the right-hand side of \eqref{bound on n} reduces to
  \[
      \frac{C p  \cdot \log p}{ \d^2} \cdot
      \log^3 \left( \frac{C p  \cdot \log p}{ \d^2 }\right)
  \]
  From the other side, if the matrix $X$ satisfies the conditions of Theorem \ref{thm::sparse-bounded-entries} with, say, $\d=1/2$, then all rows of the matrix $A$ should be present among the rows of the matrix $X$. An elementary calculation shows that in this case it is necessary to assume that $n \ge C p \log p$, so the estimate \eqref{bound on n} is exact up to a power of the logarithm.

 Unlike the matrix $\S$, the matrix $A$ is not symmetric. However, the example above can be easily modified by considering a $2p \times 2p$ matrix
 \[
   \tilde{A}= \left(
                \begin{matrix}
                  0 & A \\
                  A^T & 0 \\
                \end{matrix}
              \right).
 \]
 This shows that the estimate  \eqref{bound on n} is tight under the symmetry assumption as well.

  \end{remark}

\bibliography{subgaussian}

\begin{thebibliography}{48}
\expandafter\ifx\csname natexlab\endcsname\relax\def\natexlab#1{#1}\fi
\expandafter\ifx\csname url\endcsname\relax
  \def\url#1{\texttt{#1}}\fi
\expandafter\ifx\csname urlprefix\endcsname\relax\def\urlprefix{URL }\fi

\bibitem[{Adamczak et~al.(2011)Adamczak, Latala, Litvak, , Pajor and
  Tomczak-Jaegermann}]{ALLPT11}
\textsc{Adamczak, R.}, \textsc{Latala, R.}, \textsc{Litvak, A.~E.}, ,
  \textsc{Pajor, A.} and \textsc{Tomczak-Jaegermann, N.} (2011).
\newblock Geometry of log-concave ensembles of random matrices and approximate
  reconstruction.
\newblock 1103.0401v1.

\bibitem[{Adamczak et~al.(2009)Adamczak, Litvak, , Pajor and
  Tomczak-Jaegermann}]{ALPT09}
\textsc{Adamczak, R.}, \textsc{Litvak, A.~E.}, , \textsc{Pajor, A.} and
  \textsc{Tomczak-Jaegermann, N.} (2009).
\newblock Restricted isometry property of matrices with independent columns and
  neighborly polytopes by random sampling.
\newblock 0904.4723v1.

\bibitem[{Baraniuk et~al.(2008)Baraniuk, Davenport, DeVore and Wakin}]{BDDW08}
\textsc{Baraniuk, R.~G.}, \textsc{Davenport, M.}, \textsc{DeVore, R.~A.} and
  \textsc{Wakin, M.~B.} (2008).
\newblock A simple proof of the restricted isometry property for random
  matrices.
\newblock \textit{Constructive Approximation} \textbf{28} 253--263.

\bibitem[{Bickel et~al.(2009)Bickel, Ritov and Tsybakov}]{BRT09}
\textsc{Bickel, P.~J.}, \textsc{Ritov, Y.} and \textsc{Tsybakov, A.~B.} (2009).
\newblock Simultaneous analysis of {L}asso and {D}antzig selector.
\newblock \textit{The Annals of Statistics} \textbf{37} 1705--1732.

\bibitem[{Bunea et~al.(2007)Bunea, Tsybakov and Wegkamp}]{BTW07c}
\textsc{Bunea, F.}, \textsc{Tsybakov, A.} and \textsc{Wegkamp, M.} (2007).
\newblock Sparsity oracle inequalities for the {L}asso.
\newblock \textit{The Electronic Journal of Statistics} \textbf{1} 169--194.

\bibitem[{Cai et~al.(2010)Cai, Wang and Xu}]{CWX09}
\textsc{Cai, T.}, \textsc{Wang, L.} and \textsc{Xu, G.} (2010).
\newblock Stable recovery of sparse signals and an oracle inequality.
\newblock \textit{IEEE Transactions on Information Theory} \textbf{56}
  3516--3522.

\bibitem[{Cand\`{e}s and Plan(2009)}]{CP09}
\textsc{Cand\`{e}s, E.} and \textsc{Plan, Y.} (2009).
\newblock Near-ideal model selection by .1 minimization.
\newblock \textit{Annals of Statistics} \textbf{37} 2145--2177.

\bibitem[{Cand\`{e}s et~al.(2006)Cand\`{e}s, Romberg and Tao}]{CRT06}
\textsc{Cand\`{e}s, E.}, \textsc{Romberg, J.} and \textsc{Tao, T.} (2006).
\newblock Stable signal recovery from incomplete and inaccurate measurements.
\newblock \textit{Communications in Pure and Applied Mathematics} \textbf{59}
  1207--1223.

\bibitem[{Cand\`{e}s and Tao(2005)}]{CT05}
\textsc{Cand\`{e}s, E.} and \textsc{Tao, T.} (2005).
\newblock Decoding by {L}inear {P}rogramming.
\newblock \textit{IEEE Trans. Info. Theory} \textbf{51} 4203--4215.

\bibitem[{Cand\`{e}s and Tao(2006)}]{CT06}
\textsc{Cand\`{e}s, E.} and \textsc{Tao, T.} (2006).
\newblock Near optimal signal recovery from random projections: {U}niversal
  encoding strategies?
\newblock \textit{IEEE Trans. Info. Theory} \textbf{52} 5406--5425.

\bibitem[{Cand\`{e}s and Tao(2007)}]{CT07}
\textsc{Cand\`{e}s, E.} and \textsc{Tao, T.} (2007).
\newblock The {D}antzig selector: statistical estimation when p is much larger
  than n.
\newblock \textit{Annals of Statistics} \textbf{35} 2313--2351.

\bibitem[{Chen et~al.(1998)Chen, Donoho and Saunders}]{Chen:Dono:Saun:1998}
\textsc{Chen, S.~S.}, \textsc{Donoho, D.~L.} and \textsc{Saunders, M.~A.}
  (1998).
\newblock Atomic decomposition by basis pursuit.
\newblock \textit{SIAM Journal on Scientific and Statistical Computing}
  \textbf{20} 33--61.

\bibitem[{Donoho(2004)}]{Donoho:04}
\textsc{Donoho, D.} (2004).
\newblock For most large underdetermined systems of equations, the minimal
  $\ell_1$-norm near-solution approximates the sparsest near-solution.
\newblock Tech. rep., Stanford University.

\bibitem[{Donoho(2006{\natexlab{a}})}]{Donoho:cs}
\textsc{Donoho, D.} (2006{\natexlab{a}}).
\newblock Compressed sensing.
\newblock \textit{IEEE Trans. Info. Theory} \textbf{52} 1289--1306.

\bibitem[{Donoho(2006{\natexlab{b}})}]{Donoho06}
\textsc{Donoho, D.} (2006{\natexlab{b}}).
\newblock For most large underdetermined systems of equations, the minimal
  $\ell_1$-norm solution is also the sparsest solution.
\newblock \textit{Communications in Pure and Applied Mathematics} \textbf{59}
  797--829.

\bibitem[{Donoho and Johnstone(1994)}]{Donoho:94}
\textsc{Donoho, D.~L.} and \textsc{Johnstone, I.~M.} (1994).
\newblock Ideal spatial adaptation by wavelet shrinkage.
\newblock \textit{Biometrika} \textbf{81} 425--455.

\bibitem[{Duncan and Pearson(1991)}]{duncan:91}
\textsc{Duncan, G.} and \textsc{Pearson, R.} (1991).
\newblock Enhancing access to microdata while protecting confidentiality:
  Prospects for the future.
\newblock \textit{Statistical Science} \textbf{6} 219--232.

\bibitem[{Gordon(1985)}]{Gor85}
\textsc{Gordon, Y.} (1985).
\newblock Some inequalities for gaussian processes and applications.
\newblock \textit{Israel Journal of Mathematics} \textbf{50} 265--289.

\bibitem[{Greenshtein and Ritov(2004)}]{GR04}
\textsc{Greenshtein, E.} and \textsc{Ritov, Y.} (2004).
\newblock Persistency in high dimensional linear predictor-selection and the
  virtue of over-parametrization.
\newblock \textit{Bernoulli} \textbf{10} 971--988.

\bibitem[{Koltchinskii(2009{\natexlab{a}})}]{Kol07}
\textsc{Koltchinskii, V.} (2009{\natexlab{a}}).
\newblock {D}antzig selector and sparsity oracle inequalities.
\newblock \textit{Bernoulli} \textbf{15} 799--828.

\bibitem[{Koltchinskii(2009{\natexlab{b}})}]{Kol09}
\textsc{Koltchinskii, V.} (2009{\natexlab{b}}).
\newblock Sparsity in penalized empirical risk minimization.
\newblock \textit{Ann. Inst. H. Poincare Probab. Statist.} \textbf{45} 7--57.

\bibitem[{Ledoux(2001)}]{L01}
\textsc{Ledoux, M.} (2001).
\newblock \textit{The concentration of measure phenomenon}.
\newblock Mathematical Surveys and Monographs, 89. American Mathematical
  Society.

\bibitem[{Meinshausen and B\"{u}hlmann(2006)}]{MB06}
\textsc{Meinshausen, N.} and \textsc{B\"{u}hlmann, P.} (2006).
\newblock High dimensional graphs and variable selection with the {L}asso.
\newblock \textit{Annals of Statistics} \textbf{34} 1436--1462.

\bibitem[{Meinshausen and Yu(2009)}]{MY09}
\textsc{Meinshausen, N.} and \textsc{Yu, B.} (2009).
\newblock {L}asso-type recovery of sparse representations for high-dimensional
  data.
\newblock \textit{Annals of Statistics} \textbf{37} 246--270.

\bibitem[{Mendelson et~al.(2007)Mendelson, Pajor and
  Tomczak-Jaegermann}]{MPT07}
\textsc{Mendelson, S.}, \textsc{Pajor, A.} and \textsc{Tomczak-Jaegermann, N.}
  (2007).
\newblock Reconstruction and subgaussian operators in asymptotic geometric
  analysis.
\newblock \textit{Geometric and Functional Analysis} \textbf{17} 1248--1282.

\bibitem[{Mendelson et~al.(2008)Mendelson, Pajor and
  Tomczak-Jaegermann}]{MPT08}
\textsc{Mendelson, S.}, \textsc{Pajor, A.} and \textsc{Tomczak-Jaegermann, N.}
  (2008).
\newblock Uniform uncertainty principle for bernoulli and subgaussian
  ensembles.
\newblock \textit{Constructive Approximation} \textbf{28} 277--289.

\bibitem[{Milman and Schechtman(1986)}]{MS86}
\textsc{Milman, V.~D.} and \textsc{Schechtman, G.} (1986).
\newblock \textit{Asymptotic Theory of Finite Dimensional Normed Spaces.
  Lecture Notes in Mathematics 1200}.
\newblock Springer.

\bibitem[{Pisier(1981)}]{Pisi81}
\textsc{Pisier, G.} (1981).
\newblock Remarques sur un r\'{e}sultat non publi\'{e} de b. {M}aurey.
\newblock \textit{Seminar on Functional Analysis, \'{E}´cole Polytech.,
  Palaiseau} .

\bibitem[{Raskutti et~al.(2009)Raskutti, Wainwright and Yu}]{RWY09}
\textsc{Raskutti, G.}, \textsc{Wainwright, M.} and \textsc{Yu, B.} (2009).
\newblock Minimax rates of estimation for high-dimensional linear regression
  over $\ell_q$-balls.
\newblock In \textit{Allerton Conference on Control, Communication and
  Computer}.
\newblock Longer version in arXiv:0910.2042v1.pdf.

\bibitem[{Raskutti et~al.(2010)Raskutti, Wainwright and Yu}]{RWY10}
\textsc{Raskutti, G.}, \textsc{Wainwright, M.} and \textsc{Yu, B.} (2010).
\newblock Restricted nullspace and eigenvalue properties for correlated
  gaussian designs.
\newblock \textit{Journal of Machine Learning Research}  2241--2259.

\bibitem[{Rauhut et~al.(2008)Rauhut, Schnass and Vandergheynst}]{RSV08}
\textsc{Rauhut, H.}, \textsc{Schnass, K.} and \textsc{Vandergheynst, P.}
  (2008).
\newblock Compressed sensing and redundant dictionaries.
\newblock \textit{{IEEE} {T}ransactions on {I}nformation {T}heory} \textbf{54}
  2210--2219.

\bibitem[{Rudelson and Vershynin(2005)}]{RV05}
\textsc{Rudelson, M.} and \textsc{Vershynin, R.} (2005).
\newblock Geometric approach to error correcting codes and reconstruction of
  signals.
\newblock \textit{International Mathematical Research Notices}  4019--4041.

\bibitem[{Rudelson and Vershynin(2006)}]{RV06}
\textsc{Rudelson, M.} and \textsc{Vershynin, R.} (2006).
\newblock Sparse reconstruction by convex relaxation: Fourier and gaussian
  measurements.
\newblock In \textit{40th Annual Conference on Information Sciences and Systems
  (CISS 2006)}.

\bibitem[{Rudelson and Vershynin(2008)}]{RV08}
\textsc{Rudelson, M.} and \textsc{Vershynin, R.} (2008).
\newblock On sparse reconstruction from fourier and gaussian measurements.
\newblock \textit{Communications on Pure and Applied Mathematics}  1025--1045.

\bibitem[{Tibshirani(1996)}]{Tib96}
\textsc{Tibshirani, R.} (1996).
\newblock Regression shrinkage and selection via the {L}asso.
\newblock \textit{J. Roy. Statist. Soc. Ser. B} \textbf{58} 267--288.

\bibitem[{van~de Geer and Buhlmann(2009)}]{GB09}
\textsc{van~de Geer, S.} and \textsc{Buhlmann, P.} (2009).
\newblock On the conditions used to prove oracle results for the lasso.
\newblock \textit{Electronic Journal of Statistics} \textbf{3} 1360--1392.

\bibitem[{van~de Geer(2008)}]{vandeG08}
\textsc{van~de Geer, S.~A.} (2008).
\newblock High-dimensional generalized linear models and the {L}asso.
\newblock \textit{The Annals of Statistics} \textbf{36} 614--645.

\bibitem[{Vershynin(2011{\natexlab{a}})}]{Ver11a}
\textsc{Vershynin, R.} (2011{\natexlab{a}}).
\newblock Approximating the moments of marginals of high dimensional
  distributions.
\newblock \textit{Annals of Probability, to appear} .

\bibitem[{Vershynin(2011{\natexlab{b}})}]{Ver11b}
\textsc{Vershynin, R.} (2011{\natexlab{b}}).
\newblock How close is the sample covariance matrix to the actual covariance
  matrix?
\newblock \textit{Journal of Theoretical Probability, to appear} .

\bibitem[{Wainwright(2009)}]{Wai09}
\textsc{Wainwright, M.} (2009).
\newblock Sharp thresholds for high-dimensional and noisy sparsity recovery
  using $\ell_1$-constrained quadratic programming.
\newblock \textit{IEEE Trans. Inform. Theory} \textbf{55} 2183--2202.

\bibitem[{Zhang and Huang(2008)}]{ZH08}
\textsc{Zhang, C.-H.} and \textsc{Huang, J.} (2008).
\newblock The sparsity and bias of the lasso selection in high-dimensional
  linear regression.
\newblock \textit{Annals of Statistics} \textbf{36} 1567--1594.

\bibitem[{Zhao and Yu(2006)}]{ZY06}
\textsc{Zhao, P.} and \textsc{Yu, B.} (2006).
\newblock On model selection consistency of {L}asso.
\newblock \textit{Journal of Machine Learning Research} \textbf{7} 2541--2567.

\bibitem[{Zhou(2009{\natexlab{a}})}]{Zhou09c}
\textsc{Zhou, S.} (2009{\natexlab{a}}).
\newblock Restricted eigenvalue conditions on subgaussian random matrices.
\newblock ArXiv:0904.4723v2.

\bibitem[{Zhou(2009{\natexlab{b}})}]{Zhou09a}
\textsc{Zhou, S.} (2009{\natexlab{b}}).
\newblock Thresholding procedures for high dimensional variable selection and
  statistical estimation.
\newblock In \textit{Advances in Neural Information Processing Systems 22}. MIT
  Press.

\bibitem[{Zhou(2010)}]{Zhou10}
\textsc{Zhou, S.} (2010).
\newblock Thresholded lasso for high dimensional variable selection and
  statistical estimation.
\newblock ArXiv:1002.1583v2.

\bibitem[{Zhou et~al.(2009{\natexlab{a}})Zhou, Lafferty and Wasserman}]{ZLW09}
\textsc{Zhou, S.}, \textsc{Lafferty, J.} and \textsc{Wasserman, L.}
  (2009{\natexlab{a}}).
\newblock Compressed and privacy sensitive sparse regression.
\newblock \textit{IEEE Transactions on Information Theory} \textbf{55}
  846--866.

\bibitem[{Zhou et~al.(2011)Zhou, R\"{u}timann, Xu and B\"{u}hlmann}]{ZRXB10}
\textsc{Zhou, S.}, \textsc{R\"{u}timann, P.}, \textsc{Xu, M.} and
  \textsc{B\"{u}hlmann, P.} (2011).
\newblock High-dimensional covariance estimation based on {G}aussian graphical
  models.
\newblock \textit{Journal of Machine Learning Research, to appear} .

\bibitem[{Zhou et~al.(2009{\natexlab{b}})Zhou, van~de Geer and
  B\"{u}hlmann}]{ZGB09}
\textsc{Zhou, S.}, \textsc{van~de Geer, S.} and \textsc{B\"{u}hlmann, P.}
  (2009{\natexlab{b}}).
\newblock Adaptive {L}asso for high dimensional regression and gaussian
  graphical modeling.
\newblock ArXiv:0903.2515.

\end{thebibliography}

\end{document}